\documentclass{siamart190516}

\usepackage{amssymb,amsfonts,amsopn}
\usepackage{graphicx,url,xcolor}

\usepackage{cleveref}

\usepackage{tikz}
\usetikzlibrary{shapes.misc}

\newtheorem*{openprob}{Open Problem}

\newtheorem{assumption}{Assumption}

\newtheorem{defn}{Definition}

\DeclareMathOperator{\spec}{\text{Spec}}

\DeclareMathOperator{\R}{\mathbb{R}}

\newcommand{\pmat}[1]{\begin{pmatrix} #1 \end{pmatrix}}

\title{A Graph Spectral Flow for Computing Nodal Deficiencies}

\author{Wesley Hamilton\thanks{Department of Mathematics, University of North Carolina at Chapel Hill, Chapel Hill, NC (\email{wham@live.unc.edu}).}}

\headers{A Graph Spectral Flow}{Wesley Hamilton}

\begin{document}
	
	\maketitle
	
	\begin{abstract}
		In this paper we propose a spectral flow for graph Laplacians, and prove that it counts the number of nodal domains for a given Laplace eigenvector. This extends work done for Laplacians on $\mathbb{R}^n$ to the graph setting. We mention some open problems relating the topology of a graph to the analytic behaviour of its Laplace eigenvectors, and include numerical examples illustrating our flow.
	\end{abstract}
	
	

	\section{Introduction}
	\label{section:intro}
	
	The goal of this paper is to show that if $\psi$ is an eigenvector of the graph Laplacian $L$ then the number of nodal domains of $\psi$ can be counted via a family of perturbed graph Laplacians, the ideas of which we outline below. This provides a direct graph analogue of the continuum version in which Laplace eigenfunctions are considered, and provides an alternative proof of the nodal domain counts obtained by Berkolaiko \cite{Be:NodalMag} and Colin de Verdiere \cite{dV:MagPaper} using magnetic flux methods. 
	
	Given a connected, weighted graph $G = (V,E,w)$ with adjacency matrix $A = (w_{ij})_{i,j \in V}$, we define the graph Laplacian as $L = D - A$ where $D = (\sum_{(i,j) \in E} w_{ij})_{ii}$. The edge weights are non-negative, and zero edge weights are understood to mean the absence of an edge. For an eigenvalue/eigenvector pair $(\lambda, \psi)$ of $L$, we define the nodal domains of $\psi$ to be the maximally connected subgraphs induced by the vertex set $\{i \colon \psi_i \psi_j > 0 \text{ for some } j\in V\}$; these are often called strong nodal domains in the spectral graph theory literature. We denote the number of nodal domains of $\psi$ by $\nu(\psi)$. We show that $\nu(\psi)$ can be computed by constructing a real-parameter family of bilinear forms $B_\sigma$ on $G$ using the eigenvector $\psi$, and then considering the spectrum of $B_\sigma$ as $\sigma$ increases from $0$ to a given limit point. The number of eigenvalues of $B_\sigma$ that do not cross $\lambda$ as $\sigma$ increases is exactly the number of nodal domains.
	
	In this paper we give two constructions of $B_\sigma$. The first (\Cref{section:edge-based_flow}) is simpler to describe and is defined on the original graph, but introduces exotic new edge weights for effective Dirichlet boundary conditions on said graph. The second (\Cref{section:vertex-based_flow}) is more involved and results in a graph with extra vertices and similar exotic edge weights, but gives the motivation for the simpler construction and suggests some interesting open problems relating the zeros of an eigenvector to its nodal domains. We describe both constructions, but formulate our main result in terms of the simpler version:
	

	\begin{theorem}
		\label{theorem:main_theorem}
		Suppose $(\lambda_k, \psi)$ is the $k$th eigenvalue/eigenvector pair of $L$, $\lambda_k$ is simple, and that $\psi$ is non-zero at each vertex. Define  $$ B_\sigma(u,v) = \langle u, Lv\rangle +  \sigma \langle u, \sum_{(i,j) \in E_\pm} P_{ij} v\rangle$$ where $E_{\pm} = \{(i,j)\colon \psi_i \psi_j < 0\}$, $ P_{ij} = w_{ij} \pmat{q_{ji} & 1 \\ 1 & q_{ij}}$, and $q_{ij} = -\frac{\psi_i}{\psi_j}$. Then as $\sigma \to 1$, 
		\begin{enumerate}
			\item there are $k - \nu(\psi)$ eigenvalues of $B_\sigma$ which cross $\lambda_k$, and
			\item the number of eigenvalues of $B_\sigma$ that converge to $\lambda_k$ is exactly the number of nodal domains $\nu(\psi)$ of $\psi$.
		\end{enumerate}
	\end{theorem}
	
	Part 2 is the content of \cref{theorem:edge-nodal_thm,theorem:vertex-main_theorem}, from which Part 1 is a straightforward corollary.
	
	In the rest of this section, we provide the context and motivation for this result in the continuum and graph settings. Namely, we start by reviewing what is known about nodal domains and nodal deficiencies for Laplace eigenfunctions, and then discuss similar graph eigenvector results. 
	
	\subsection{The Continuum Spectral Flow}
	\label{section:intro-continuum}
	
	Consider a connected, bounded domain $\Omega \subset \R^n$ with Lipschitz boundary. The eigenvalues of the Laplacian $\Delta$ restricted to $\Omega$, with Dirichlet boundary conditions, form an increasing sequence $0< \lambda_0 \leq \lambda_1 \leq \cdots$; call their corresponding eigenfunctions $\phi_0, \phi_1, ...$ The \textbf{nodal sets} of an eigenfunction $\phi_k$ are the connected components of $\{\phi_k  = 0\} =: \Gamma$, the \textbf{nodal domains} are the connected components of $\Omega \setminus \Gamma$, and the number of nodal domains is denoted $\nu(\phi_k)$. The \textbf{nodal deficiency} of an eigenfunction $\phi_k$ corresponding to a simple eigenvalue $\lambda_k$ is defined as $$\delta(\phi_k) = k - \nu(\phi_k);$$ if $\lambda_k$ is not simple, we set $k_* = \inf\{s \colon \lambda_s = \lambda_k \}$ and define $$\delta(\phi_{k}) = k_* - \nu(\phi_{k}).$$ Below we will see that $\delta(\phi_k)\geq 0$ for all $k$.
	
	When $n=1$ and $\Omega$ is a bounded, connected interval, the classical Sturm-Liouville theory states that the nodal deficiency is always $0$:
	
	\begin{theorem}
		\label{theorem:sturm-liouville}
		Let $\Omega = [0,1]$ and consider the Dirichlet eigenvalue problem $$\partial_{xx}u(x) = \lambda u(x) \text{ for } x\in(0,1) \text{ and } u(0) = u(1) = 0.$$ Sort the eigenvalues $0<\lambda_0 \leq \lambda_1 \leq \cdots$, and call the corresponding eigenfunctions $\phi_0, \phi_1, ...$ Then $\phi_k$ has exactly $k$ zeros in $(0,1)$.
	\end{theorem}
	For a modern discussion of this result, see \cite[Chapter XIII]{ReSi:Book4}.
	
	In higher dimensions the situation is significantly more difficult. One early result was Courant's nodal theorem, which provides an upper bound on the nodal deficiency:
	
	\begin{theorem}
		\label{theorem:Courant_nodal}
		Let $\Omega\subset \R^n, n\geq 2$, be a bounded, connected domain with Laplacian $\Delta$, and let $0<\lambda_1 \leq \lambda_2 \leq \cdots$ be the ordered eigenvalues for the Dirichlet eigenvalue problem $$\begin{cases} \Delta u = \lambda u & \text{ in }\Omega, \\ u = 0 & \text{ on } \partial \Omega.\end{cases}$$
		If $\phi_1, \phi_2, ...$ are the associated eigenfunctions, then $k \geq \nu(\phi_k)$.
	\end{theorem}
	
	For a full proof of this theorem see \cite[Chapter 1.5]{Ch:EigBook}; for more on Dirichlet eigenvalue problems, see \cite[Chapter 6.4]{Ev:PDEBook}. We mention, in particular, a corollary of Courant's nodal theorem:
	
	\begin{proposition}[{\cite[Cor. 2]{Ch:EigBook}}]
		\label{proposition:dirichlet_eigenfunctions}
		With the same terminology as in \cref{theorem:Courant_nodal},
		\begin{itemize}
			\item $\phi_1$ has constant sign;
			\item $\lambda_1$ has multiplicity $1$;
			\item $\lambda_1$ is characterized as being the only eigenvalue with eigenfunction of constant sign.
		\end{itemize}
	\end{proposition}
	
	
	The upper bound in Courant's theorem can only be attained finitely many times, as implied by Pleijel's results in \cite{Pl:Nodal}. On the other hand, there exist eigenfunctions with arbitrarily large index that have few nodal domains. One procedure to construct such eigenfunctions is the following: let $\Omega = [0,\pi]\times [0,\pi]$ and let $\psi_1 = \sin(k_1x) \sin(l_1y),\psi_2 = \sin(k_2x) \sin(l_2y)$ be two eigenfunctions of $\Delta$ with Dirichlet boundary conditions, such that $k_1^2 + l_1^2 = k_2^2 + l_2^2$; $k_1 = l_2 \neq k_2 = l_1$ is one such choice. Consider the 1-parameter family of eigenfunctions $\psi_t = (1-t)\psi_1 + t \psi_2$ for $t\in[0,1]$. As $t$ varies from $0$ to $1$, the nodal domains of $\psi_1$ will merge and transform until they align with the nodal domains of $\psi_2$ when $t=1$. Depending on the choice of $k_i, l_i$ the number of nodal domains of $\psi_t$ for $t\in (0,1)$ may get as low as 2, but in general will be significantly fewer than the number of nodal domains for $\psi_1$ or $\psi_2$ for $k_i, l_i$ large. 
	
	This discussion suggests that counting nodal deficiencies is in general difficult, even in low dimensions. A step towards resolving these difficulties is presented in \cite{BeMaCo:SpecFlow}, where the nodal deficiency is reinterpreted as the Morse index of the Dirichlet-to-Neumann operator. Through this interpretation, the authors are able to count the nodal deficiency as the spectral flow of a bilinear form that combines the Dirichlet energy for a domain with a kind of Dirac mass on the eigenfunction's nodal line. Their main result is the following.
	
	\begin{theorem}
		\label{theorem:continuum_flow}
		The nodal deficiency of $\phi_k$ is precisely the number of eigenvalues of the bilinear form $$B_\sigma(u,v) := \int_{\Omega} \nabla u\cdot \nabla v  d\mu + \sigma \int_{\Gamma} uv dS$$ that cross $\lambda_k +\epsilon$ for $\epsilon>0$ sufficiently small, as $0\leq \sigma \to \infty$. Here $\Gamma = \{x \colon \phi_k(x) = 0\}\cap \Omega$, the nodal set of $\phi_k$ in the interior of the domain. Equivalently, the number of nodal domains of $\phi_k$ is exactly the multiplicity of the first Dirichlet eigenvalue on $\Omega\setminus \Gamma$, which are precisely the eigenvalues of the limiting bilinear form $B_\infty$.
	\end{theorem}
	
	The proof is straightforward after the right framework is introduced, and many of the results and proofs in this paper are direct graph analogues of the above continuum result. In our formulation, the domain $\Omega$ is replaced by a weighted graph $G$, $\Delta$ is replaced by the graph Laplacian $L$ on $G$, and the nodal set $\Gamma$ is replaced by the edges over which a graph Laplacian eigenvector changes sign. \Cref{theorem:main_theorem} shows that our graph spectral flow is able to count the nodal deficiency of a graph Laplacian eigenvector, just as the continuum spectral flow counts the nodal deficiency for a Laplace eigenfunction.
	
	Naturally, one could ask what happens when the graph spectral flow is constructed on a graph built from a point cloud sampled from $\Omega \subset \R^n$; as we sample more points and construct ``denser and denser'' graphs, do the graph spectral flows converge to the continuum spectral flow? This question will be the subject of future work.
	
	\subsection{Nodal deficiencies of graph Laplacian eigenvectors}
	\label{section:intro-nodal_deficiencies_graphs}
	
	One of the implicit themes in this work is connecting analytic properties of the graph Laplacian to topological properties of the underlying graph. Similar ideas can be found in some of the early work of Fiedler (see, for example, \cite[Theorem (2,3)]{Fi:AcycMat}). Here we highlight the monograph \cite{BiLeSt:EigvecBook}, along with some more recent work due to Berkolaiko in \cite{Be:LowerBound},\cite{Be:NodalMag}, and shortly after by Colin De Verdiere \cite{dV:MagPaper}. 
	
	While graph Laplacians have been studied since the 20th century, nodal domain theorems for graph Laplacians have appeared relatively recently;  \cite{BiLeSt:EigvecBook} contains a fairly complete overview of what is currently known. Since graph functions are discrete, the notion of ``nodal set'' requires a little more care. Given a graph $G = (V, E, w)$ and a vector $\psi\in \R^{|V|}$ (interpreted as a function $\psi \colon V \to \R$), two kinds of nodal domains are defined. The first are the \textbf{weak nodal domains}, which are maximally connected subgraphs corresponding to the edge sets $\{ (i,j) \colon w_{ij}>0 \text{ and } \psi_i\psi_j \geq 0  \}$; these are precisely the connected subgraphs where $\psi$ takes the same sign, including zero, on each vertex. There are also the \textbf{strong nodal domains}, which are maximally connected subgraphs corresponding to the edge sets $\{ (i,j) \colon w_{ij}>0 \text{ and } \psi_i\psi_j > 0  \}$. One of the main results in the literatures is
	
	\begin{theorem}[{\cite[Theorem 3.1]{BiLeSt:EigvecBook},\cite[Theorem 2]{davies.2001.LAA}}]
		\label{theorem:graph_nodal_counts}
		For any graph $G$, the $k$th eigenvector $\psi_k$ of the graph Laplacian $L$ has at most $k$ weak nodal domains and at most $k + r - 1$ strong nodal domains, where $r$ is the multiplicity of $\lambda_k$.
	\end{theorem}
	Commonly found proofs utilize matrix-theoretic methods, and are actually stated for a larger class of operators called generalized graph Laplacians. In our work we focus on strong nodal domains, though our vertex-based flow can be used to count weak nodal domains directly. Moreover, our methods have a distinct spectral theoretic flavour, due to the continuum analogue our construction is based on.
	
	We next highlight recent contributions by Berkolaiko \cite{Be:NodalMag} and Colin de Verdiere \cite{dV:MagPaper}, whose proofs are closer in spirit to the current work. Given a graph $G = (V,E)$, define the $1$st Betti number $\beta_1$ of $G$ to be the number of linearly independent cycles in $G$; this number can be interpreted in the sense of simplicial homology, or as the minimum number of edges that need to be removed from $E$ to turn $G$ into a tree. Also suppose $G$ has a \textbf{magnetic field}, which is a function $B\colon \overrightarrow{E\,} \to \mathbb{C}$ that satisfies $B_{ji}= \overline{B_{ij}}$ and $|B_{ij}| = 1$ for all $(i,j)\in \overrightarrow{E\,}$, where $\overrightarrow{E\,}$ is the collection of oriented edges of $G$. We construct a magnetic Laplacian $L_B$ on $G$ through the quadratic form $$q_B(f) = \frac{1}{2} \sum_{(i,j)\in \overrightarrow{E}} w_{ij} |f_i - e^{B_{ij} \sqrt{-1}} f_j|^2 - \sum_{i \in V} V_i |f_i|^2,$$ where $V_i = \sum_{(i,j)\in E} w_{ij}$. Note that $L_B$ is Hermitian and so has real spectrum $$\lambda_1(L_B) < \lambda_2(L_B) \leq \cdots \leq \lambda_{|V|}(L_B).$$ See  \cite{dV:MagPaper} for more on this construction.
	
	Suppose $\psi$ is an eigenvector associated to the $k$th eigenvalue $\lambda_k$ of $L_B$. \begin{theorem}[{\cite[Theorem 1.1]{Be:NodalMag}}]
		\label{theorem:magnetic_counts}
		If $\lambda_k$ is simple and $\psi$ is never zero, then the number of edges $Z_{\psi}$ over which $\psi$ changes sign satisfies $k-1\leq Z_{\psi} \leq k-1+\beta_1$.
		
		Moreover, the nodal deficiency $\delta(\psi) = Z_{\psi} - \nu(\psi)$ is the Morse index (number of negative eigenvalues) of the operator $\Lambda_k \colon B \to \lambda_k (L_B)$, and $\Lambda_k$ is smooth at its critical point $B\equiv 1$. 
	\end{theorem}
	
	A corollary of this theorem is that, under the same assumptions, $k - \beta_1 \leq \nu(\psi) \leq k$, where $\nu(\psi)$ is the number of nodal domains of the graph function $\psi$. As mentioned above proofs of these results can be found in \cite{Be:LowerBound,Be:NodalMag,dV:MagPaper}. This paper provides an alternative proof of these upper bounds utilizing Dirichlet eigenvalues of graphs built from the original graph Laplacian, without requiring the use of magnetic fields.
	
	\subsection{Organization of paper}
	\label{section:intro-organization}
	
	\Cref{section:edge-based_flow} gives an edge-based graph spectral flow construction, which is used to give an alternative proof of Courant's nodal theorem for graphs. \Cref{section:vertex-based_flow} gives an alternative vertex-based graph spectral flow; while no new results are established with this other flow, it does suggest interesting connections between an eigenvector's sign-change edges and its nodal domains. Finally, \Cref{section:examples} provides some numerical examples illustrating the behaviour of the edge-based and vertex-based spectral flow for a number of graphs, including Erd\'{o}s-Renyi graphs for a range of probabilities.
	
	\subsection{Acknowledgments}
	\label{section:intro-acknowledgements}
	
	Thanks to J.L. Marzuola for suggesting this problem, guidance through the research process, and careful readings of preliminary versions of this paper. Thanks to G. Berkolaiko for suggesting this problem on graphs at an AIMS meeting, as well as H.T. Wu for conversations on the numerical implementation. Thanks to anonymous referees for extensive and helpful suggestions. The author was supported by NSF CAREER Grant DMS-1352353 and NSF Applied Math Grant DMS-1909035, as well as the Thelma Zaytoun Summer Research Fellowship from the UNC-CH Graduate School.
	
	\section{The edge-based spectral flow}
	\label{section:edge-based_flow}
	
	In this section we define the edge-based graph spectral flow and use it to prove that the $k$th eigenvector of a graph Laplacian has no more than $k$ nodal domains. \Cref{section:edge-based_flow-notation} states assumptions used throughout this section and establishes notation. \Cref{section:edge-based_flow-construction,section:edge-based_flow-nodal_count} give the edge-based spectral flow construction and establish some of its properties, including the main result of this paper.
	
	\subsection{Notation and assumptions}
	\label{section:edge-based_flow-notation}
	
	Suppose $G = (V,E,w)$ is a weighted graph without multiple edges. Vertices will generally be denoted by natural numbers, edges will be $2$-tuples of vertices and will be denoted as either $(i,j)$, $e_{ij}$, or just $e$, and edge weights will be written $w(e), w((i,j)),$ or $w_{ij}$; $w(e) = 0$ means the edge $e$ is not present in the graph. We only consider graphs with non-negative edge weights. The adjacency matrix of $G$ is the $|V|\times |V|$ matrix $W = (w_{ij})_{(i,j)\in E}$, and the degree matrix is the diagonal $|V|\times |V|$ matrix $D = (\sum_j w_{ij})_{i \in V}$. The spectrum of $G$ will be the spectrum of its graph Laplacian $L = D-W$; in particular, we are not considering the normalized graph Laplacian in this paper, nor are we considering the spectrum of adjacency matrices. For more on graph Laplacians and their spectra, see \cite{Ch:SpecBook} or \cite{BrHa:SpecBook}.
	
	Given a graph $G = (V,E,w)$, its graph Laplacian $L$ is positive semi-definite and so its spectrum consists of real eigenvalues $0 = \lambda_1 \leq \lambda_2 \leq \cdots \leq \lambda_n$. For each $k = 1, 2, ...$, we consider the eigenvalue/eigenvector pair $(\lambda_k, \psi)$. We make two assumptions on the pairs $(\lambda_k, \psi)$ throughout this paper:
	
	\begin{assumption}
		\label{assump:multiplicity}
		In case $\lambda_k$ has multiplicity greater than 1, $k$ will be the first index for which $\lambda_k$ appears in the spectrum, i.e. $k = \min\{ m \colon  L\psi = \lambda_m \psi \}$.
	\end{assumption}

	\begin{assumption}
		\label{assump:non-zero}
		The eigenvector $\psi$ is non-zero on each vertex of $G$, which turns out to be a generic property of graph Laplacians; see the introduction of \cite{Be:NodalMag} for an extended discussion.
	\end{assumption}

	This first assumption ensures that our bounds are not affected by eigenvalue multiplicities. 
	
	The second assumption greatly simplifies notation and the ensuing arguments, and can always be enforced by (1) perturbing the graph Laplacian by a diagonal matrix $Q$, or (2) perturbing the edge weights to ``shift'' a zero off of a vertex. Note that this perturbation may significantly change the number of weak and strong nodal domains: suppose four strong nodal domains/two weak nodal domains meet in an ``X'' with the center of the ``X'' a zero vertex. Performing the aformentioned perturbation will cause two of the strong nodal domains to merge, while splitting one of the weak nodal domains into two separate components. In the interest of completeness we also mention the necessary modifications when $\psi$ does have zeros on $G$, though the results and proofs are the same.
	
	\subsection{The construction}
	\label{section:edge-based_flow-construction}
	
	Given an eigenvalue/eigenvector pair $(\lambda_k,\psi)$ of the graph Laplacian $L$, define the sign change edges $E_{\pm} = \{(i,j) \colon \psi_i \psi_j < 0 \}$. For each $(i,j)\in E_\pm$, define the rank-1, $|V|\times |V|$ matrices $P_{ij} = w_{ij} \pmat{q_{ji} & 1 \\ 1 & q_{ij}}$ with $q_{ij} := -\frac{\psi_i}{\psi_j}$ and zeros everywhere else. 
	
	\begin{defn}
		\label{def:edge_based_flow}
		We define the \textbf{edge-based spectral flow} as the collection of eigenvalues associated to the family of bilinear forms 
	
	\begin{align*}
		B_{\sigma}(u,v) :&= \langle u, L v \rangle + \sigma \langle u, \sum_{(i,j)\in E_\pm} P_{ij} v \rangle\\
			&= \langle u, \left(L+\sigma \sum_{(i,j)\in E_\pm} P_{ij} \right) v \rangle.
	\end{align*}
	We set $P =\sum_{(i,j)\in E_\pm} P_{ij}$ and  $L_\sigma = L+\sigma P$, so that $B_\sigma(u,v) = \langle u , L_\sigma v\rangle$. The edge-based spectral flow is the curve $(\lambda_1(\sigma), \lambda_2(\sigma),...,\lambda_{|V|}(\sigma))$ with $\sigma\in[0,1]$ where each $\lambda_i$ is an eigenvalue branch of $B_\sigma$.
	\end{defn}
	A similar procedure works when $\psi$ has zeros on $G$: construct $L_\sigma$ as above, and then delete the rows and columns corresponding to the zeros of $\psi$. This procedure is the construction of a Dirichlet graph Laplacian associated to zeros on a graph, which we revisit in \Cref{section:vertex-based_flow-relation_to_Dirichlet}. For simplicity we keep \cref{assump:non-zero}.
	
	We are interested in the \textbf{nodal domains} of $\psi$, which are the connected components of the subgraph of $G$ induced by the edge set $E\setminus E_\pm$ on the vertex set $V$. Explicitly, this induced subgraph is $G_\psi := (V, E\setminus E_\pm, w_{\psi})$, where $w_\psi(e) := w(e)$ for $e\in E\setminus E_\pm$ and $w_\psi((i,i)) := \sum_{(i,j)\in E_\pm} (1+q_{ji}) w_{ij}$. The value of $w_\psi$ on self-loops/vertices keeps track of those edges that cross into different nodal domains, and imposes effective Dirichlet boundary conditions across those edges. The choice of $w_{ij}(1+q_{ji})$, versus just $w_{ij}$, in the definition of $w_\psi$ is motivated by \cref{lemma:edge-lambda_k_constant,lemma:vertex-lambda_k_constant}: this choice of edge weight ensures that $\psi$ is still an eigenvector of $B_\sigma$ corresponding to $\lambda_k$, and allows us to construct a basis of $\lambda_k$ eigenvectors of $G_{\psi}$ by restricting $\psi$ to the connected components of $G_{\psi}$. We mention that, by \cref{assump:non-zero}, the nodal domains we consider are what are called strong nodal domains in the literature \cite{BiLeSt:EigvecBook}.
	
	The \textbf{number of nodal domains} $\nu(\psi)$ of $\psi$ is the number of connected components of $G_\psi$, and the \textbf{nodal deficiency} $\delta(\psi)$ is $\delta(\psi) = k - \nu(\psi)$. The main result of this section relates the eigenvalues of $L_\sigma$ to $\nu(\psi)$:
	
	\begin{theorem}
		\label{theorem:edge-nodal_thm}
		The nodal domain count $\nu(\psi)$ is the multiplicity of $\lambda_k$ in the spectrum of $L_1 = L +  P$, and the nodal deficiency of $\psi$ is precisely the number of eigenvalue branches that cross $\lambda_k$.
	\end{theorem}

	\begin{corollary}
		\label{corollary:edge-nodal_count}
		The nodal count satisfies $\nu(\psi) \leq k$, where $\psi$ is any eigenvector of $\lambda_k$.
	\end{corollary}

	\subsection{The nodal domain count and the edge-based spectral flow}
	\label{section:edge-based_flow-nodal_count}

	The proof of \cref{theorem:edge-nodal_thm} relies on a series of Lemmas that describe the behaviour of the eigenvalue and eigenvector branches. \cref{lemma:edge-lambda_k_constant,lemma:edge-lambda_sigma_increasing} are straightforward computations that show, respectively, the $\lambda_k$ branch is constant, and that the eigenvalue branches are non-decreasing in $\sigma$. \cref{lemma:edge-dirichlet_eigvecs} is the key conceptual and technical piece, which establishes that restricting $\psi$ to each of its nodal domains provides a basis for the $\lambda_k$ eigenspace of $L_1$. 

	\begin{lemma}
		\label{lemma:edge-lambda_k_constant}
		The eigenvalue $\lambda_k$ is in the spectrum of $L_\sigma$ for $0\leq \sigma \leq 1$, and $L_\sigma \psi = \lambda_k \psi$. In particular the $\lambda_k$ eigenvalue branch is constant.
	\end{lemma}

	\begin{proof}
		By construction, $\psi$ is in the kernel of each $P_{ij}$: $$P_{ij} \psi = w_{ij}\pmat{q_{ji} \psi_i + \psi_j \\ \psi_i + q_{ij} \psi_j} = w_{ij} \pmat{ -\psi_j + \psi_j \\ \psi_i - \psi_i} = \pmat{0 \\ 0}.$$ Thus, $$L_\sigma \psi =L\psi + \sigma  P \psi =  L\psi= \lambda_k \psi.$$
	\end{proof}

	\begin{lemma}
		\label{lemma:edge-lambda_sigma_increasing}
		The eigenvalues of $B_\sigma$, which are the eigenvalues of the matrix $L_\sigma$, are non-decreasing eigenvalue branches in $\sigma$ for $0<\sigma <1$ .
	\end{lemma}

	\begin{proof}
		Suppose $(\lambda, u) = (\lambda_\sigma, u_\sigma)$ is an eigenvalue/eigenvector pair of $B_\sigma$ with $\langle u, u\rangle = 1$, so that $$B_\sigma(u,v) = \lambda \langle u, v\rangle \quad \forall v \in \R^{|V|}.$$ Each $\lambda_\sigma$ is an analytic curve in $\sigma$, branching from the eigenvalue/eigenvector pairs of $L_0 = L$; this follows from standard perturbation theory \cite{Ka:PertBook}. Differentiating with respect to $\sigma$ gives
		\begin{align*}
			B_\sigma(u,v)' = B_\sigma'(u,v) + B_\sigma(u',v), \text{ and } 
			(\lambda \langle u, v \rangle)' 	= \lambda' \langle u, v\rangle + \lambda \langle u', v\rangle.
		\end{align*}
		
		By the variational formulation for eigenvalues we must have $B_\sigma(u,u') = \lambda \langle u, u'\rangle$, and so
		\begin{align*}
			\lambda' \langle u, u\rangle + \lambda \langle u', u\rangle = B_\sigma'(u,u) + B_\sigma(u',u),
		\end{align*}
		which in turn gives 
		\begin{align*}
			\lambda' = B_\sigma'(u,u) = \langle u, L_{\sigma}' u \rangle = \left\langle u, P u \right\rangle.
		\end{align*}
	
		Now $\langle u, P_{ij} u\rangle = (\sqrt{q_{ji}} u_i + \sqrt{q_{ij}}u_j)^2 \geq 0$, so $\left\langle u, P u \right\rangle \geq 0$ and $\lambda' \geq 0$ as desired.
	\end{proof}

	Let $L_\psi$ be the graph Laplacian of $G_\psi$, i.e. $L_\psi = D - A$ with $A = (w_{\psi,ij})_{(i,j)\in E\setminus E_\pm}$ and $D = (\sum_{(i,j)\in E} w_{ij})_{ii}$. As mentioned in the definition of $G_\psi$, $L_\psi$ is effectively a Dirichlet graph Laplacian for $G_\psi$ for which the Dirichlet boundary condition is imposed across the edges in $E_\pm$. Note that $L_\psi = L_1$, which can be seen by writing out the entries of each matrix explicitly. By construction $G_\psi$ consists of $\nu(\psi)$ connected components, so we write $G_\psi = G_1 \cup G_2 \cup \cdots \cup G_{\nu(\psi)}$ where $G_i \cap G_j = \emptyset$ for $i\neq j$. The next lemma describes the spectrum of $L_\psi$ through the eigenvectors of $L_\psi$ restricted to each $G_i$. 

	\begin{lemma}
		\label{lemma:edge-dirichlet_eigvecs}
		The spectrum $0 < \lambda^\psi_1 \leq \lambda^\psi_2 \leq \cdots \leq \lambda^\psi_{|V|}$ of $L_\psi$ consists of:
		\begin{enumerate}
			\item $\lambda^{\psi}_1 = \cdots = \lambda^{\psi}_{\nu(\psi)} = \lambda_k$, and
			\item $\lambda^{\psi}_{\nu(\psi)+1} > \lambda_k$.
		\end{enumerate}
		
		Restricting $\psi$ to each nodal domain $G_i$ gives a signed eigenvector of $\lambda^{\psi}_1$, so the eigenspace of $L_\psi$ for $\lambda^{\psi}_1$ is the span of $\psi|_{G_1}, ..., \psi|_{G_{\nu(\psi)}}$. Moreover, eigenvectors of higher eigenvalues must be signed on each connected component of $G_\psi$.
	\end{lemma}

	This result is a direct graph analogue of the theorem for Dirichlet eigenvalues for the Laplacian acting on a connected, bounded domain; see \cite[\S 1.5, Corollary 2]{Ch:EigBook}. See also \cref{lemma:vertex-first_eigvals_D-connected,lemma:vertex-Dirichlet_eigenvectors} for an extended discussion of Dirichlet eigenvalues in the graph setting.
	
	\begin{proof}
		Claim 1. is presented in \cite[Lemma 6.1]{BiLeSt:EigvecBook} using the Dirichlet eigenvalue framework, with their $V^\circ$ corresponding to each of our connected components $G_i$, and their vertex boundary $\partial V$ corresponding to the sets $\{l \colon \exists j \text{ s.t. } (m,n)\in E_\pm \text{ and } m \in G_i \}$ in our case. Since the first Dirichlet eigenvalue of each $G_i$ is simple, and there are $\nu(\psi)$ such $G_i$, we must have $\lambda^{\psi}_{\nu(\psi)+1} > \lambda^\psi_1$.
		
		To produce $\nu(\psi)$ explicit eigenvectors for $\lambda^\psi_1$, we can restrict $\psi$ to each $G_i$. Let $\psi|_{G_i}$ denote the vector $\psi$ with entries on $G_i^c$ all zero. Then a straightforward computation gives, for each $j \in G_i$,
		
		\begin{align*}
			(L_\psi \psi|_{G_i})_j &= \sum_{(j,m) \in E} w_{\psi,jm} \psi_j - \sum_{(j,m) \in E\setminus E_\pm} w_{\psi,jm}\psi_m\\
					&= \sum_{(j,m) \in E_\pm} w_{\psi,jm} \psi_j +  \sum_{(j,m) \in E\setminus E_\pm} w_{\psi,jm} (\psi_j - \psi_m)\\
					&= \sum_{(j,m) \in E_\pm} w_{jm} (1+q_{mj}) \psi_j +  \sum_{(j,m) \in E\setminus E_\pm} w_{jm} (\psi_j - \psi_m)\\
					&= \sum_{(j,m) \in E_\pm} w_{jm} (\psi_j - \psi_m)+  \sum_{(j,m) \in E\setminus E_\pm} w_{jm} (\psi_j - \psi_m)\\
					&= \sum_{(j,m) \in E} w_{jm} (\psi_j - \psi_m)\\
					&= \lambda_k \psi_j,
		\end{align*}
		since $\psi$ is an eigenvector of $L$. Thus $L_\psi \psi|_{G_i} = \lambda_k \psi|_{G_i} = \lambda^{\psi}_1 \psi|_{G_i} $, and since each $\lambda^{\psi}_1$ is simple, eigenvectors of higher eigenvalues must be orthogonal to each $\psi|_{G_i}$ and hence must be signed.
	\end{proof}
	
	\begin{lemma}
		\label{lemma:edge-lambda_k_crossings}
		Let $(\lambda_\sigma, u)$ be an eigenvalue/eigenvector pair of $L_\sigma$ for $0\leq \sigma \leq 1$, where $u$ depends on $\sigma$. If $\lambda_{\sigma^*}' = 0$ for some $\sigma^*$ then $\lambda_\sigma$ is constant and in the spectrum of $L_1$. Moreover if $\lambda_{\sigma^*} = \lambda_k$ then we also have that $u$ is a multiple of $\psi$.
	\end{lemma}

	In practice, \cref{lemma:edge-lambda_k_crossings} is used to show that eigenvalue branches that cross $\lambda_k$ must cross with a positive slope, and hence limit to an eigenvalue strictly larger than $\lambda_k$.

	\begin{proof}
		Recall that $$\lambda' = \langle u, P u \rangle = \sum_{(i,j) \in E_\pm} w_{ij}(\sqrt{q_{ji}} u_i + \sqrt{q_{ij}} u_j)^2.$$ If $\lambda'=0$ then $u_i = \frac{\psi_i}{\psi_j} u_j$ for each $(i,j)\in E_\pm$, and $Pu = 0$. But then $L_{\sigma^*} u = Lu = \lambda_{\sigma^*} u$, so $\lambda_{\sigma^*}$ is in the spectrum of $L$ with $u$ a corresponding eigenvector. Thus $\lambda_{\sigma}$ is constant on the interval $[0,\sigma^*]$, and since these eigenvalue branches are analytic, $\lambda_\sigma$ is constant and in the spectrum of $L_1$.
		
		If moreover $\lambda_{\sigma^*} = \lambda_k$, by \cref{lemma:edge-dirichlet_eigvecs} $u$ is a linear combination of the restrictions $\psi|_{G_k}$ with $G_k$ a nodal domain of $\psi$. For a fixed $G_k$, we can find a constant $\alpha$ such that $u|_{G_k} = \alpha \psi|_{G_k}$. The condition $u_i = \frac{\psi_i}{\psi_j} u_j$ for each $(i,j)\in E_{\pm}$ shows $u|_{G_l} = \alpha \psi|_{G_l}$ whenever $G_l$ and $G_k$ are connected by a sign-change edge. Since $G$ is connected, $u = \alpha \psi$ on all of $G$.
	\end{proof}

	With the pieces all in place, the proof of \cref{theorem:edge-nodal_thm} is straightforward.

	\begin{proof}[Proof of \cref{theorem:edge-nodal_thm}]
		By \cref{lemma:edge-lambda_sigma_increasing} the eigenvalue branches of $L_\sigma$ are non-decreasing, and so are either constant or strictly increasing by \cref{lemma:edge-lambda_k_crossings}. \cref{lemma:edge-dirichlet_eigvecs} tells us that precisely $\nu(\psi)$ eigenvalue branches of $L_\sigma$ converge to $\lambda_k$, so $\delta(\psi) = k - \nu(\psi)$ of the eigenvalues below $\lambda_k$ will cross $\lambda_k$ with positive slope and hence converge to eigenvalues strictly greater then $\lambda_k$.
	\end{proof}

	\cref{theorem:edge-nodal_thm} suggests a means to compute nodal domains and nodal deficiencies: given $\psi$ with $L\psi = \lambda_k \psi$, construct $L_1 = L + \sum_{(i,j)\in E_{\pm}} P_{ij}$ and compute the multiplicity of $\lambda_k$ in the spectrum of $L_1$. While this is sufficient for applications to data analysis, we are also interested in consistency aspects of the graph spectral flow. In particular suppose $X_n$ is a point cloud of $n$ points sampled from a manifold $M$ with respect to a measure $\mu$, and $G_{n,\epsilon}$ is a geometric graph with edges between points that are at most distance $\epsilon$ apart. A natural question is whether our graph spectral flow converges to the continuum spectral flow appearing in \cref{theorem:continuum_flow}. One immediate concern is the parameter range: the graph spectral flow is defined for $\sigma \in [0,1]$, while the continuum spectral flow is defined for $\sigma \in [0,\infty)$. The consistency of this graph spectral flow to the continuum version is the subject of future work, and is motivated by the vertex-based spectral flow constructed next.

	\section{The vertex-based spectral flow}
	\label{section:vertex-based_flow}
	
	In this section we outline a vertex-based graph spectral flow. This flow has the same properties as the edge-based flow, but relies on ``ghost vertices'' and ``ghost edges'' added to the graph. These ``ghosts'' sit where effective zeros should be expected to be found along the sign-change edges. This approach is useful for two reasons:
	\begin{enumerate}
		\item basis vectors for the Dirichlet Laplacian's first eigenvalue originate as indicator functions of ghost points, suggesting an interesting interplay between an eigenvector's zeros and nodal domains;
		\item incorporating the zeros of an eigenvector as vertices of the graph makes establishing consistency of the flow more straightforward, in part because the limit in $\sigma$ is taken to $\infty$ instead of $1$.
	\end{enumerate} 

	In \cref{section:vertex-based_flow-construction} we outline the vertex-based graph spectral flow and state the analogous results to the edge-based flow. This flow relies on a new graph we call the $\psi$-subdivision, where $\psi$ is a Laplace eigenvector. The following subsection makes explicit the relation between the vertex-based flow and graph Laplacians with Dirichlet boundary conditions, while the last section connects this flow to the edge-based flow discussed above.
	
	\subsection{The construction and properties}
	\label{section:vertex-based_flow-construction}
	
	\begin{defn}
		\label{def:vertex-psi_subdivision}
		Given an eigenvector $\psi$ of the graph Laplacian we define
		\begin{itemize}
			\item the \textbf{sign-change edges} $E_\pm \subset E$ as those edges $(i,j)$ such that $\psi_i \psi_j < 0$;
			\item the \textbf{ghost vertices} $V_{gh} = \{0_{ij} \colon (i,j) \in E_\pm \}.$
		\end{itemize}
		
		The $\psi$\textbf{-subdivision graph} $G_{\psi,\sigma}$ of $G$ is the new graph $$G_{\psi,\sigma} = (V_\psi, E_{\psi},w_{\psi,\sigma}),$$ depending on a parameter $\sigma \in [0,\infty)$, with
		\begin{itemize}
			\item $V_\psi := V \cup  V_{gh}$,
			\item $E_\psi := E \cup \{ (i, 0_{ij}), (0_{ij},j) \}_{(i,j) \in E_0}$, and
			\item $w_{\psi, \sigma}(e) = \begin{cases} w(e), & e \in E\setminus E_\pm,\\ \frac{1}{1+\sigma} w(e), & e\in E_\pm,\\ \frac{\sigma}{1+\sigma} w(\tilde{e}) (1+q_{ji}), & e = (i,0_{ij}), \tilde{e} = (i,j), \, q_{ji} := \frac{-\psi_j}{\psi_i}>0. \end{cases}$
		\end{itemize}	
		Finally, we write $L_{\psi,\sigma}$ for the graph Laplacian of $G_{\psi,\sigma}$. 
	\end{defn}

	The idea behind the $\psi$-subdivision graph is to add vertices and edges that explicitly incorporate the zeros of $\psi$ into the graph structure. \cref{pic:K2GPsi} shows the subdivision process for the complete graph on 2 vertices, $K_2$, with $\psi = (1,-1)$. A new ghost vertex $0_{12}$ is added halfway between vertices 1 and 2 approximately where a zero on the edge (1,2) would occur. The edge weights are chosen so that as $\sigma \to \infty$, the original edge (1,2) dissappears and the edges adjacent to $0_{12}$ have the correct edge-weights to impose Dirichlet boundary conditions on each nodal domain of $\psi$.
	
	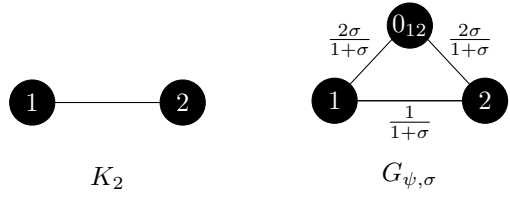
\begin{figure}
		\centering
		\begin{tikzpicture}
			\draw (0,0) -- (2,0);
			\draw [black, fill = black] (0,0) circle [radius = 0.3];
			\draw [black, fill = black] (2,0) circle [radius = 0.3];
			
			\node [white] at (0,0) {$1$};
			\node [white] at (2,0) {$2$};
			\node at (1,-1) {$K_2$};
		\end{tikzpicture}\qquad\qquad
		\begin{tikzpicture}
			\draw (0,0) -- (2,0);
			\draw [rounded corners] (0,0) -- (2,0) -- (1,1.07) -- (0,0);
			\draw [black, fill = black] (0,0) circle [radius = 0.3];
			\draw [black, fill = black] (2,0) circle [radius = 0.3];
			\draw [black, fill = black] (1,1) circle [radius = 0.3];
			
			\node [white] at (0,0) {$1$};
			\node [white] at (2,0.) {$2$};
			\node [white] at (1.,1) {$0_{12}$};
			\node at (1,-1) {$G_{\psi,\sigma}$};
			
			\node at (1.0,-0.3) {$\frac{1}{1+\sigma}$};
			\node at (0.2,1-0.2) {$\frac{2\sigma}{1+\sigma}$};
			\node at (1.8,1-0.2) {$\frac{2\sigma}{1+\sigma}$};
		\end{tikzpicture}
		
		\caption{$K_2$ and its subdivision graph $G_{\psi,\sigma}$ for $\psi=(1,-1)$. All unmarked edges have edge weight 1.}
		\label{pic:K2GPsi}
	\end{figure}

	Note that $L_{\psi,0}$ is the $(|V|+|V_{gh}|)\times (|V| + |V_{gh}|)$ matrix with $L$ in the $|V|\times |V|$ upper-left block and zeros elsewhere.
	
	Our first result shows that if $\psi$ is an eigenvector of $L$, then $\psi$ is also an eigenvector of $L_{\psi,\sigma}$. To make this precise, we need to extend the vector $\psi\in \R^{|V|}$ to a vector $\tilde{\psi}\in \R^{|V| + |V_{gh}|}$ and show that $L_{\psi,\sigma} \tilde{\psi} = \lambda_k \tilde{\psi}$. By construction of $G_{\psi,\sigma}$ we expect $\tilde{\psi}|_{V_{gh}} = 0$, though we are also interested in extending other vectors $u\in \R^{|V|}$ to vectors $\tilde{u} \in \R^{|V| + |V_{gh}|}$.

	\begin{defn}
		\label{def:vertex-extension}
		A vector $f\in \R^{|V|}$, interpreted as a function on $G$, can be extended to $\tilde{f} \in \R^{|V| + |V_{gh}|}$, interpreted as a function on $G_{\psi,\sigma}$, by setting $\tilde{f}_i = f_i$ for $i\in V$, and $\tilde{f}_{0_{ij}} = a_{ij} f_i + a_{ji} f_j$ for $0_{ij}\in V_{gh}$ with $a_{ij} = \frac{1}{1+q_{ij}}.$
	\end{defn}  
	Note that $\tilde{\psi}_{0_{ij}} = a_{ij} \psi_i + a_{ji} \psi_j	= \frac{\psi_i\psi_j}{\psi_j - \psi_i} + \frac{\psi_i\psi_j}{\psi_i - \psi_j} = 0,$ as desired.

	\begin{lemma}
		\label{lemma:vertex-lambda_k_constant}
		Suppose $(\lambda_k,\psi)$ is an eigenvalue/eigenvector pair for the graph $G$, i.e. $L \psi = \lambda_k \psi$. Then $L_{\psi, \sigma} \tilde{\psi} = \lambda_k \tilde{\psi}$ for all $\sigma$.
	\end{lemma}

	\begin{proof}
		This is a straightforward computation. Because $\psi$ is an eigenvector with eigenvalue $\lambda_k$, we have
		\begin{align*}
			(L \psi)_i 	&= \sum_{(i,j)\in E} w_{ij} (\psi_i - \psi_j) = \lambda_k \psi_i.
		\end{align*}
		If the vertex $i$ is not in $V_{gh}$, then
		\begin{align*}
			(L_{\psi,\sigma} \tilde{\psi})_i =& \sum_{(i,j) \in E_\psi} w_{ij,\sigma} (\psi_i - \psi_j) \\
			=& \sum_{(i,j) \in E} w_{ij,\sigma} (\psi_i - \psi_j) + \sum_{(i,0_{ij})\in E_\psi \setminus E} \frac{\sigma}{1+\sigma} w_{ij} (1+q_{ji}) (\psi_i - \psi_{0_{ij}})\\
			=& \sum_{(i,j) \in E\setminus E_\pm} w_{ij} (\psi_i - \psi_j)\\
			&+ \sum_{(i,j)\in E_\pm} w_{ij} \left[ \frac{1}{1+\sigma} (\psi_i - \psi_j) + \frac{\sigma}{1+\sigma} (1+q_{ji}) \psi_i \right]\\
			=& \sum_{(i,j) \in E\setminus E_\pm} w_{ij} (\psi_i - \psi_j)\\
			&+ \sum_{(i,j)\in E_\pm} w_{ij} \left[ \frac{1}{1+\sigma} (\psi_i - \psi_j) + \frac{\sigma}{1+\sigma} (\psi_i - \psi_j) \right]\\
			=& \sum_{(i,j)\in E} w_{ij} (\psi_i - \psi_j) = \lambda_k \psi_i	=\lambda_k \tilde{\psi}_i.
		\end{align*}
		Otherwise,
		\begin{align*}
			(L_{\psi,\sigma} \tilde{\psi})_{0_{ij}} 	&= \frac{\sigma}{1+\sigma} w_{ij} (1+q_{ji}) (\psi_{0_{ij}} - \psi_i) + \frac{\sigma}{1+\sigma} w_{ij} (1+q_{ij}) (\psi_{0_{ij}} - \psi_j)\\
			&= \frac{-\sigma w_{ij}}{1+\sigma} ((1+q_{ji})\psi_i + (1+q_{ij})\psi_j)\\
			&= 0 = \lambda_k \tilde{\psi}_{0_{ij}},
		\end{align*}
		and so $L_{\psi,\sigma} \tilde{\psi} = \lambda_k \tilde{\psi}$.
	\end{proof}

	\begin{defn}
		\label{def:vertex-bilinear_form}
		Define the family of bilinear forms $B_\sigma$ on $G_\psi$ by $$B_\sigma(u,v) = \langle u, L_{\psi,\sigma} v\rangle + \sigma \langle u, v\rangle_{V_{gh}}.$$ Here, $\langle u, v\rangle_{V_{gh}}$ is the inner product for $G_\psi$ restricted to $V_{gh}$. Written out in full,
		\begin{align*}
			B_\sigma(u,v) 	=& \sum_{(i,j)\in E\setminus E_\pm} w_{ij}(u_i - u_j)(v_i-v_j)\\
			&+ \sum_{(i,j)\in  E_\pm } w_{ij} \frac{1}{1+\sigma} (u_i - u_j)(v_i - v_j)\\
			&+ \sum_{(i,j) \in E_\pm} w_{ij} \frac{\sigma}{1+\sigma} \bigg[ (1+q_{ji})(u_i - u_{0_{ij}})(v_i - v_{0_{ij}})  \\
			&\qquad\qquad\qquad\qquad + (1+q_{ij})(u_j - u_{0_{ij}})(v_j - v_{0_{ij}}) \bigg]\\
			&+ \sigma \sum_{i\in V_{gh}} u_i v_i.
		\end{align*}
	\end{defn}
	We are again assuming that $\psi$ is non-zero on each vertex of $G$. If $\psi$ does have zeros, the corresponding bilinear form is $ B_\sigma(u,v) = \langle u, L_{\psi,\sigma} v\rangle + \sigma \langle u, v\rangle_{V_{gh}} + \sigma\langle u,v\rangle_{V_0},$ where $V_0 = \{i \colon \psi_i =0\}$. All of the results in this section still hold, and so for simplicity we keep making use of \cref{assump:non-zero}.

	\begin{lemma}
		\label{lemma:vertex-lambda_sigma_increasing}
		The eigenvalues of $B_\sigma$ are non-decreasing eigenvalue branches of the eigenvalues of $L_{\psi,0}$, for $0<\sigma<\infty$.
	\end{lemma}

	\begin{proof}
		The proof is the same as in the edge-based flow case from \cref{lemma:edge-lambda_sigma_increasing}: we have $$\lambda' = B_\sigma'(u,u) = \langle u, L_{\psi, \sigma}' u \rangle + \langle u,u\rangle_{V_{gh}},$$ and $$\langle u, L_{\psi, \sigma}' u\rangle = \sum_{(i,j)\in E_\pm} \frac{w_{ij}}{(1+\sigma)^2} q_{ij}(u_{0_{ij}} + q_{ji} u_{0_{ij}} - q_{ji} u_i - u_j)^2,$$ the latter of which is a straightforward computation. Since $w_{ij}, q_{ij}$ are both non-negative we conclude that $\lambda' \geq 0$.
 	\end{proof}

	\subsection{The relation to Dirichlet Laplacians}
	\label{section:vertex-based_flow-relation_to_Dirichlet}
	
	Our results on the graph spectral flow involve the limiting behaviour of $B_\sigma$ and $L_{\psi,\sigma}$ as $\sigma \to \infty$. For such a statement like $L_{\psi,\infty} u = \lambda u$ to make sense, we need $u|_{ V_{gh}} = 0$; for the rest of this paper, we use the convention that $0\cdot \infty = 0$. Thus, the limiting eigenvalue problem asks for a function $u$ with $L_{\psi,\infty} u = \lambda u$ and $u|_{ V_{gh}} = 0$. This is reminiscent of a Dirichlet boundary value condition, so we begin by recalling the basic definitions and properties of Dirichlet eigenvalues for graphs. Afterwards we return to the vertex-based graph spectral flow, and finish the proof that this flow counts the nodal deficiency of a graph eigenvector. For a complete introduction to Dirichlet eigenvalues on graphs, see \cite[Chapter 8]{Ch:SpecBook}.
	
	\begin{defn}
		\label{def:vertex-boundaries}
		For a graph $G = (V,E,w)$ and a subset of vertices $S$, we define:
		\begin{itemize}
			\item the \textbf{vertex boundary} $\partial_V S$ as the vertices in $V\setminus S$ that are adjacent to some vertex in $S$, and
			\item the \textbf{edge boundary} $\partial_E S$ as the edges in $E$ that connect a vertex in $\partial_V S$ to a vertex in $S$.
		\end{itemize}
		The space of vectors $u \in \R^{|V|}$ that are zero on $\partial_V S\subset V$ is denoted $D^*_S$ or just $D^*$ when $S$ is clear, i.e. $$D^* = \{u\in \R^{|V|} \colon u|_S = 0\}.$$
		
		Finally, the \textbf{Dirichlet subgraph induced by $S$}, or the D-subgraph induced by $S$, denoted $S^{(D)}$, is the subgraph of $G$ induced by the vertices in $S$, together with the vertices of $\partial_V S$ and edges of $\partial_E S$; explicitly, the induced subgraph is  $(S\cup \partial_V S, E|_{S} \cup \partial_E S, w|_{E|_{S} \cup \partial_E S})$.
	\end{defn}

	This notion of vertex boundaries allows us to impose Dirichlet/zero boundary conditions on problems involving the graph Laplacian, which was implicit in the construction from \Cref{section:edge-based_flow-construction}.
	
	\begin{defn}
		\label{def:vertex-first_Dirichlet_eigenvalue}
		The \textbf{first Dirichlet eigenvalue} of a graph $G$, corresponding to $S$, is 
		\begin{align*}
			\lambda_1^{(D)} 	= \inf_{\substack{{u \neq 0}\\ {u \in D^*}}} \sum_{(i,j)\in \partial_E S} \frac{w_{ij} (u_i - u_j)^2}{\sum_{i\in S} u_i^2} 
			= \inf_{\substack{{u \neq 0}\\ {u \in D^*}\\{\langle u, u\rangle = 1}}} \sum_{(i,j)\in \partial_E S} \langle u, L^{(D)} u\rangle_S.
		\end{align*}
		The operator $L^{(D)}$ is the graph Laplacian of $G$ with the rows and columns corresponding to vertices in $V\setminus S$ removed.
		
		Higher order eigenvalues are found inductively via the Courant-Fischer/Min-max theorem (see, for example, \cite[Chapter 1, \S 10]{Ka:PertBook}): after determining $\lambda_1^{(D)}, ..., \lambda_k^{(D)},$ and associated eigenvectors $\phi_1, ..., \phi_k,$ we have 
		\begin{align*}
			\lambda_{k+1}^{(D)} 	&= \inf_{\substack{{u \neq 0}\\ {u \in D^*}\\ {\langle u, u\rangle = 1}\\{u \perp \phi_i, 1\leq i \leq k}}} \sum_{(i,j)\in E(S^{(D)})} w_{ij} (u_i - u_j)^2.
		\end{align*}
	\end{defn}

	Right away we see that $\lambda_1^{(D)} >0$. In fact, if the induced subgraph $S^{(D)}$ is connected (modulo zero vertices, to be made precise), then the corresponding eigenvector is signed. This result is used to show that the first Dirichlet eigenvalue of a connected subgraph is simple, which is then used to show that higher eigenvectors cannot be signed.

	\begin{defn}
		\label{def:vertex-D-connected}
		Given a graph $G=(V,E)$ and a subset of vertices $S$, we call the induced D-subgraph of $S$  \textbf{Dirichlet disconnected} if there are subgraphs $S_1, S_2$ of $G$ such that $S^{(D)} = S_1^{(D)} \cup S_2^{(D)}$ and $S_1\cap S_2 \subset \partial_V S$. Otherwise, $S$ is \textbf{Dirichlet connected} if $S$ is not Dirichlet disconnected and both $S_1$ and $S_2$ are connected subgraphs of $G$. We will write this last term as D-connected.
	\end{defn}

	An equivalent characterization for an induced D-subgraph $S^{(D)}$ to be D-connected is that any two vertices are path-connected in $S^{(D)},$ where the path cannot pass through $\partial_V S$.

	\begin{lemma}
		\label{lemma:vertex-first_eigvals_D-connected}
		Suppose that the subgraph $S^{(D)}$ is D-connected. Then
		\begin{enumerate}
			\item the eigenvector $\phi_1$ corresponding to $\lambda_1^{(D)}$ is signed,
			\item $\lambda_1^{(D)}$ is simple, and
			\item higher index eigenvectors $\phi_i$ cannot be signed, implying a signed eigenvector must correspond to the first Dirichlet eigenvalue.
		\end{enumerate}
	\end{lemma}

	\cref{lemma:vertex-first_eigvals_D-connected,lemma:vertex-Dirichlet_eigenvectors} together form the analogue to \cref{lemma:edge-dirichlet_eigvecs}, with the key difference being that $G_{\psi,\sigma}$ contains explicit vertices for the zeros of $\psi$. Note that \cref{lemma:vertex-first_eigvals_D-connected} is stated for the D-connected components of $G$, whereas \cref{lemma:edge-dirichlet_eigvecs} is stated for the entire graph.

	\begin{proof}
		Claims 1. and 2. are proved in \cite[Lemma 6.1]{BiLeSt:EigvecBook}, with their $V^\circ \cup \partial V$ corresponding to our $S^{(D)}$.
		
		For claim 3., since $\phi_{k}$ minimizes $\sum_{(i,j)\in E(S^{(D)})} w_{ij} (u_i - u_j)^2$ over all $u$ with $\langle u, u\rangle = 1, u \neq 0, u \in D^*$, and $u\perp \phi_i$ for $1\leq i \leq k-1$, we have in particular that $\langle \phi_k, \phi_1\rangle = 0$. We already have that $\phi_1$ is signed, and so if $\phi_k$ was signed as well, assuming both eigenvectors positive gives $\langle \phi_1, \phi_k\rangle > 0$. Thus a higher signed eigenvector cannot be orthogonal to $\phi_1$, forcing $\phi_k$ to change sign within $S$.
	\end{proof}

	\begin{proposition}
		\label{lemma:vertex-Dirichlet_eigenvectors}
		Given a graph $G$ and a nowhere zero Laplace eigenvector $\psi$ with eigenvalue $\lambda$, decompose the nodal domains $S = \{i \colon \psi_i > 0\}\cup \{i \colon \psi_i <0 \}$ of the $\psi$-subdivision $G_{\psi,\infty}$ into D-connected graphs $S_1, S_2, ..., S_n$. Then the restriction of $\psi$ to each $S_l$, $\psi|_{S_l}$, is a Dirichlet eigenvector of $S^{(D)}$ with eigenvalue $\lambda$. Moreover, $\psi|_{S_l}$ is signed, and so $\lambda$ is the first Dirichlet eigenvalue for each $S_l$.
	\end{proposition}
	
	\begin{proof}
		Recall that $G_{\psi,\infty}$ contains the original vertices of $G$ together with ghost points $0_{ij}$ for each $(i,j)\in E_\pm$, and each edge $(i,j)\in E_\pm$ is replaced by two edges $(i,0_{ij})$ and $(0_{ij},j)$, with respective edge weights $(1+q_{ji})w_{ij}$ and $(1+q_{ij})w_{ij}$. 
		
		For a D-connected component $S_l$, define
		\begin{align*}
			\psi|_{S_l} = \begin{cases} \psi_i, & i\in S_l,\\ 0, & i \not\in S_l,	\end{cases}
		\end{align*}
		which is the restriction of $\psi$ to $S_l$, followed by an extension by zero to the rest of the graph. We claim that $\psi|_{S_l}$ is an eigenvector of $L_{\psi,\infty}$ restricted to $S_l$, which implies that $\psi|_{S_l}$ is also a Dirichlet eigenvector of $S_l$.
		
		In general, for any vector $u$ that is zero on $ V_{gh}$ we have 
		\begin{align*}
			(L_{\psi,\infty} u)_i &= \sum_{(i,j) \in E\setminus E_\pm} w_{ij} (u_i - u_j) + \sum_{(i,j)\in E_\pm} w_{ij} (1+q_{ji})(u_i - u_{0_{ij}})
		\end{align*}
		For $i\in S_l$,
		\begin{align*}
			(L_{\psi,\infty} \psi|_{S_l})_i =&  \sum_{(i,j) \in E\setminus E_\pm} w_{ij} ((\psi|_{S_l})_i - (\psi|_{S_l})_j)\\
			&+ \sum_{(i,j) \in E_\pm} w_{ij} (1+ q_{ji}) ((\psi|_{S_l})_i - (\psi|_{S_l})_{0_{ij}})\\
			=& \sum_{(i,j) \in E\setminus E_\pm} w_{ij} (\psi_i - \psi_j) + \sum_{(i,j) \in E_\pm} w_{ij} (1+ q_{ji}) \psi_i\\
			=& \sum_{(i,j) \in E\setminus E_\pm} w_{ij} (\psi_i - \psi_j) + \sum_{(i,j) \in E_\pm} w_{ij} (\psi_i - \psi_j)\\
			=& \sum_{(i,j) \in E} w_{ij} (\psi_i - \psi_j) = \lambda \psi_i = \lambda (\psi|_{S_l})_i,
		\end{align*}
		where the sum over $E_\pm$ can be empty or not depending on if $i$ has neighbors in $V_{gh}$. This shows each $\psi|_{S_l}$ is a Dirichlet eigenvector of $S_l$ with eigenvalue $\lambda$. Moreover, each $S_l$ is a D-connected subgraph of $G_{\psi,\infty}$ corresponding to a nodal domain $\{i \colon \psi_i > 0\}$ or $\{i \colon \psi_i < 0 \}$, and so each $\psi|_{S_l}$ is signed.
		
		Thus we have constructed signed Dirichlet eigenvectors for $\lambda$ on each of the D-connected components of $S^{(D)}$, establishing that $\lambda$ is the first Dirichlet eigenvalue for each $S_l$.
	\end{proof}
	
	Note that in determining whether $\psi|_{S_l}$ is a Dirichlet eigenvector, we only check the eigenvalue equation within $S_l$ and not on $\partial_V S_l$; in general, $(L_{\psi,\infty} \psi|_{S_l})_i \neq \lambda (\psi|_{S_l})_i$ for $i \in \partial_V S_l$.

	\begin{lemma}
		\label{lemma:vertex-lambda_k_crossings}
		Let $(\lambda_\sigma, u_\sigma) = (\lambda,u)$ be an eigenvalue/eigenvector branch of $B_\sigma$. If $\lambda_{\sigma^*}' = 0$ for some $\sigma*\in(0,\infty)$ then the corresponding eigenvalue branch $\lambda_{\sigma}$ is constant and $\lambda_\sigma$ is in the spectrum of $L_{\psi,\infty}$. Moreover if $\lambda_{\sigma^*} = \lambda_k$ then the eigenvector $u$ is a constant multiple of $\psi$.
	\end{lemma}

	\begin{proof}
		This proof follows mutatis mutandis as in the proof of \cref{lemma:edge-lambda_k_crossings}: from $\lambda' = 0$ we have $\langle u, L_{\psi, \sigma}' u \rangle = 0$ and $\langle u,u\rangle_{V_{gh}} = 0$. The latter equality forces $u_{0_{ij}} = 0$ for $(i,j)\in E_\pm$, after with the former imposes $u_i = \frac{\psi_i}{\psi_j} u_j$ across $(i,j)\in E_\pm$.
	\end{proof}

	\begin{theorem}
		\label{theorem:vertex-main_theorem}
		As $\sigma \to \infty$, the eigenvalues of $B_\sigma$ converge to the Dirichlet eigenvalues of the D-subgraph $S^{(D)} = (\{i \colon \psi > 0\} \cup \{i \colon \psi < 0\} )^{(D)}$. The number of D-connected components of $S^{(D)}$ is the multiplicity of $\lambda_k$ for $B_\infty$, and the nodal deficiency of $\psi$ on $G_{\psi,\infty}$ is $\delta(\psi) = k - \nu(\psi)$. Note that, by construction, there will be $k - \nu(\psi)+|V_{gh}|$ eigenvalue branches that cross $\lambda_k$ as $\sigma \to \infty$.
	\end{theorem}

	\begin{proof}
		This proof follows directly as in \cref{theorem:edge-nodal_thm}.
	\end{proof}

	One feature about our vertex-based flow is that the eigenvalue branches $\lambda_\sigma$ that converge to $\lambda_k$ almost always start at zero, namely $\lambda_{\sigma = 0} = 0$; the corresponding eigenvectors originate as indicator vectors for the ghost points. Of course if there are more nodal domains than ghost vertices, then some of the non-zero eigenvalues of $L$ will also converge to $\lambda_k$.
	
	This observation suggests that the topology of a graph $G$, and in particular the collection of sign-change edges $E_\pm$ for an eigenvector $\psi$, play an important role in determining the nodal domains of $\psi$.

	\begin{openprob}
		How do the sign-change edges contribute to the nodal domain counts? For each eigenvalue branch converging to $\lambda_*$, the corresponding eigenvector will converge to a linear combination of first Dirichlet eigenvectors for each D-connected domain of $G_{\psi}$: what do the eigenvectors tell us about the nodal domains, and how does the graph topology determine which sign-change edges give rise to eigenvectors of $L_{\psi,\infty}$?
	\end{openprob}

	\subsection{The relation to the edge-based flow}
	\label{section:vertex-based_flow-relation_to_edge-based}
	
	In this short subsection we relate the vertex-based construction to the edge-based construction of \Cref{section:edge-based_flow-construction}. 
	
	\begin{proposition}
		\label{prop:vertex_to_edge_flow}
		Suppose $\tilde{u}, \tilde{v}\colon G_{\psi,\sigma} \to \R$ are extensions of functions $u,v$ on $G$. Then 
		\begin{align*}
			B_\sigma(\tilde{u},\tilde{v}) 	&= \langle u, Lv\rangle + \sigma \sum_{(i,j)\in E_\pm} \frac{a_{ij}a_{ji}}{w_{ij}} \langle u, P_{ij} v\rangle .
		\end{align*} 
	\end{proposition}

	\begin{proof}
		For functions $\tilde{u}$ and $\tilde{v}$ that are extensions of functions $u, v$ on $G$, we have $$\tilde{u}_{0_{ij}} = a_{ij}u_i + a_{ji} u_j = \frac{1}{1+q_{ij}} u_i + \frac{1}{1+q_{ji}}u_j,$$ so the term $\sigma \sum_{i\in V_{gh}} \tilde{u}_i \tilde{v}_i$ of $B_\sigma (\tilde{u},\tilde{v})$ becomes 
		
		\begin{align*}
			&\sum_{(i,j)\in E_\pm} a_{ij}^2 u_i v_i + a_{ij} a_{ji} u_i v_j  + a_{ij}a_{ji} u_j v_i + a_{ji}^2 u_j v_j\\
			=& \sum_{(i,j)\in E_\pm} a_{ij}a_{ji} (\sqrt{q_{ji}}u_i + \sqrt{q_{ji}}u_j)(\sqrt{q_{ji}}v_i + \sqrt{q_{ji}}v_j)\\
			=&\sum_{(i,j)\in E_\pm} u^T p_{ij} v.
		\end{align*}
		
		Here $p_{ij}$ is the matrix with zeros except at the $i,j$ submatrix, taking the form
		
		\begin{align*}
			p_{ij} 	&= \pmat{a_{ij}^2 & a_{ij}a_{ji}\\ a_{ij}a_{ji} & a_{ji}^2 } = a_{ij}a_{ji} \pmat{q_{ji} & 1 \\ 1 & q_{ij}} = \frac{a_{ij}a_{ji}}{w_{ij}} P_{ij}.
		\end{align*}
		
		We also see that 
		
		\begin{align*}
			&\sum_{(i,j) \in E_\pm} w_{ij} \frac{\sigma}{1+\sigma} \left[ (1+q_{ji})(u_i - u_{0_{ij}})(v_i - v_{0_{ij}}) + (1+q_{ij})(u_j - u_{0_{ij}})(v_j - v_{0_{ij}}) \right]\\
			=& \sum_{(i,j) \in E_\pm} w_{ij} \frac{\sigma}{1+\sigma} \left[ (1+q_{ji}) ((1-a_{ij})u_i - a_{ji} u_j)((1-a_{ij})v_i - a_{ji} v_j) \right.\\
			&+ \left.(1+q_{ij}) ((1-a_{ji})u_j - a_{ij}u_i)((1-a_{ji})v_j - a_{ij}v_i) \right].\\
			=& \sum_{(i,j)\in E_\pm} w_{ij} \frac{\sigma}{1+\sigma} \left[ a_{ji}(u_i - u_j)(v_i - v_j) + a_{ij} (u_j - u_i)(v_j - v_i)  \right]\\
			=& \sum_{(i,j)\in E_\pm} w_{ij} \frac{\sigma}{1+\sigma} (u_i - u_j)(v_i - v_j),
		\end{align*}
		since $a_{ij} + a_{ji} = 1$ and $\frac{a_{ij}}{1+q_{ij}} = 1$. We conclude
		
		\begin{align*}
			B_\sigma(u,v) 	=& \sum_{(i,j)\in E\setminus E_\pm} w_{ij}(u_i - u_j)(v_i-v_j) + \sum_{(i,j)\in E_\pm } w_{ij} \frac{1}{1+\sigma} (u_i - u_j)(v_i - v_j)\\
			&+ \sum_{(i,j) \in E_\pm} w_{ij} \frac{\sigma}{1+\sigma} \big[ (1+q_{ji})(u_i - u_{0_{ij}})(v_i - v_{0_{ij}}) \\
			& \qquad\qquad\qquad\qquad + (1+q_{ij})(u_j - u_{0_{ij}})(v_j - v_{0_{ij}}) \big] + \sigma \sum_{i\in V_{gh}} u_i v_i\\
			=& \sum_{(i,j)\in E\setminus E_\pm} w_{ij}(u_i - u_j)(v_i-v_j) + \sum_{(i,j)\in E_\pm } w_{ij} \frac{1}{1+\sigma} (u_i - u_j)(v_i - v_j)\\
			&+ \sum_{(i,j)\in E_\pm} w_{ij} \frac{\sigma}{1+\sigma} (u_i - u_j)(v_i - v_j) + \sigma \sum_{i\in V_{gh}} u_i v_i\\
			=& \sum_{(i,j)\in E} w_{ij}(u_i - u_j)(v_i-v_j)\\
			&+ \sigma \sum_{(i,j)\in E_\pm} a_{ij}a_{ji} (\sqrt{q_{ji}}u_i + \sqrt{q_{ji}}u_j)(\sqrt{q_{ji}}v_i + \sqrt{q_{ji}}v_j)\\
			=& \langle u, Lv\rangle + \sigma \sum_{(i,j)\in E_\pm} \frac{a_{ij}a_{ji}}{w_{ij}} \langle u, P_{ij} v\rangle.
		\end{align*}
	\end{proof}
	
	The constant in front of each $\langle u,P_{ij}v\rangle$ determines when effective Dirichlet boundary conditions are imposed across edges $(i,j)\in E_\pm$, so in general we can consider the bilinear form $\langle u, Lv\rangle + \sigma \sum_{(i,j)\in E_\pm} c_{ij} \langle u, P_{ij} v\rangle$. For the choice $c_{ij} = w_{ij}$, we see that when $\sigma=1$ the Laplacian $L_1$ indicates each edge $(i,j)\in E_\pm$ is no longer present, which is where the Dirichlet boundary conditions come from.
	
	This version of the vertex-based bilinear form requires that $\tilde{u}$ and $\tilde{v}$ are extensions of vectors $u$ and $v$, which in general may not be the case for eigenvectors of $L_{\psi,\sigma}$. Nonetheless, as $\sigma\to\infty$ the vertex-based flow forces $u|_{V_{gh}} = v|_{V_{gh}} = 0$. This leads to $0 = u_{0_{ij}} = \frac{1}{1+q_{ij}} u_i + \frac{1}{1+q_{ji}}u_j$, and so $u_i = -\frac{1+q_{ij}}{1+q_{ji}} u_j = \frac{\psi_i}{\psi_j}u_j$ as in \cref{lemma:edge-dirichlet_eigvecs}.

	\section{Examples}
	\label{section:examples}
	
	In this section we provide some examples of both the subdivision process and spectral flow for some common types of graphs. For some of these graphs we can explicitly state what the spectrum is, and we state these without proof; see \cite{BrHa:SpecBook} for details.
	
	\subsection{Complete graphs}
	\label{section:examples-complete}
	
	For a complete graph on $n$ vertices, denoted $K_n$, we label the vertices $\{1, 2, ..., n\}$ and add in all edges $(i,j)$, $1\leq i < j \leq n$. The spectrum of the graph Laplacian is $\{0, n,..., n\}$, with $n$ repeated $n-1$ times, and the (complex valued) eigenvectors are $(1,\xi, \xi^2, ..., \xi^{n-1})$ for roots of unity $\xi^n = 1$, both facts due to the graph Laplacian being circulant; see any text on matrix analysis, such as \cite[Chapter 12]{Be:MatrixBook}, for details.
	
	In \cref{fig:K5eigs}, we display the eigenvectors and spectral flows corresponding to $\lambda_2$ and $\lambda_3$. The top row shows the second Laplace eigenvector for $K_5$, followed by the edge-based and vertex-based flows. The first plot shows the eigenvector's values on each vertex. The next two plots show the spectral flow for the edge-based and vertex-based flows, respectively. We show all eigenvalue branches for the sake of illustration, but of particular note is the fact that only two of the branches converge to $5$, and the rest quickly diverge from the line $\lambda_k = 5$. For the vertex-based flow, we have eigenvalue branches corresponding to the original vertices as well as the ghost vertices; one of the two branches that limits to $5$ originated as a zero eigenvalue, corresponding to one of the ghost vertices.
	
	%
	%
	
	
	\begin{figure}
		\centering
		\includegraphics[width=0.35\textwidth]{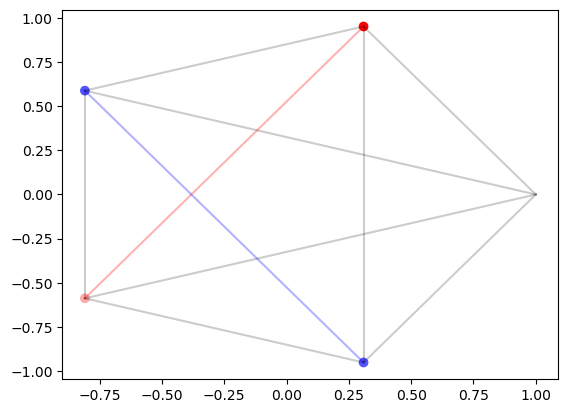}
		\includegraphics[width=0.3\textwidth]{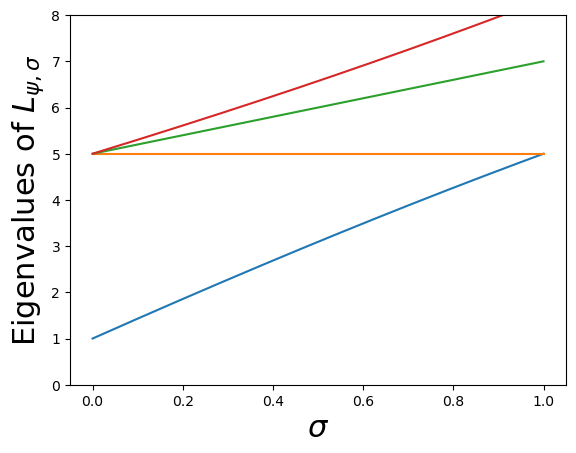}
		\includegraphics[width=0.3\textwidth]{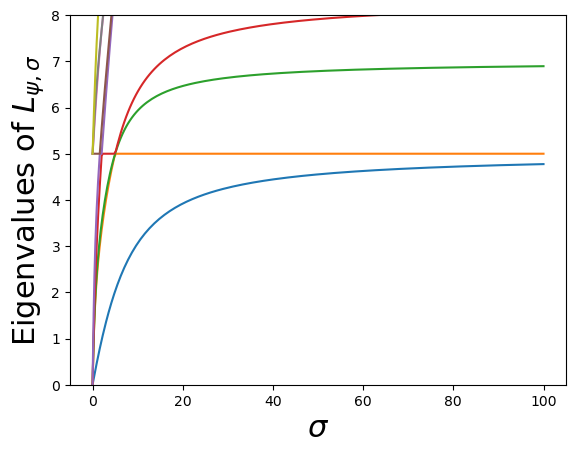}
		
		\includegraphics[width=0.35\textwidth]{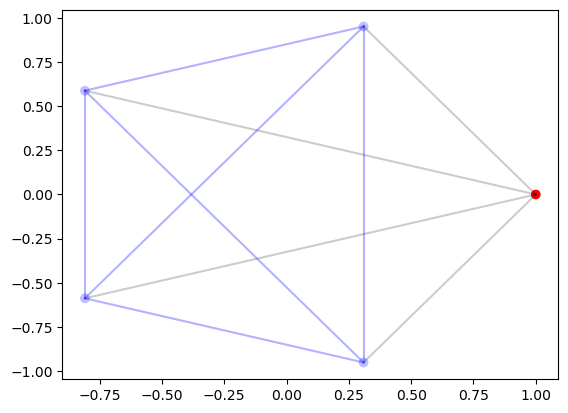}
		\includegraphics[width=0.3\textwidth]{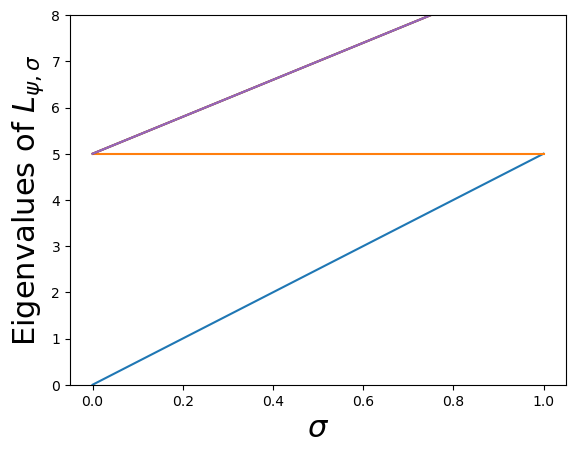}
		\includegraphics[width=0.3\textwidth]{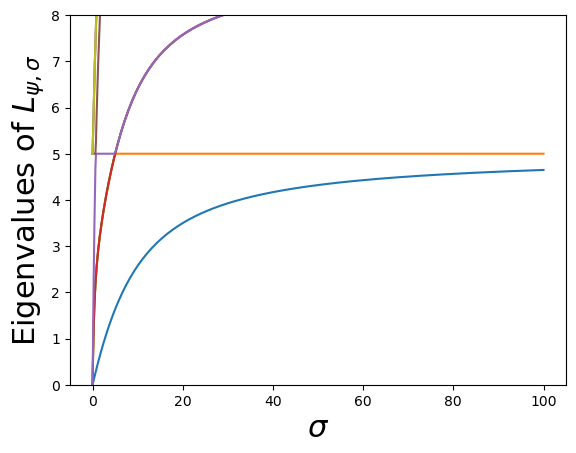}
		
		\caption{The second (top row) and third (bottom row) eigenvector of the graph Laplacian for $K_5$, along with their edge-based (middle column) and vertex-based (right column) spectral flows. In the edge-based flow for the third eigenvector (middle bottom), only three eigenvalue branches appear; the other two are hidden by the eigenvalue branch above $\lambda_3$.}
		\label{fig:K5eigs}
	\end{figure}
	
	\subsection{Cyclic graphs}
	\label{section:examples-cyclic}
	
	The cyclic graph on $n$ vertices, denoted $C_n$, has vertices $\{1, 2, ..., n\}$, and edges $(i,i+1)$ for $1\leq i < n$, and $(n,1)$. The spectrum of $C_n$ is $\{2 - 2\cos(\frac{2\pi j}{n}) \}_{j=0}^{n-1}$. Accordingly, each eigenvalue has multiplicity 2.
	
	In \cref{fig:C5eig1}, we show the second Laplace eigenvector for $C_5$. In the edge-based flow plot (middle), we show the flow of all five eigenvalues branches for $C_5$, as well as the vertex-based flow (right). Examining the function plot suggests this eigenvector has 2 nodal domains, which is verified in the spectral flows via 2 eigenvalue branches converging to the second eigenvalue.
	
	\begin{figure}
		\centering
		\includegraphics[width=0.35\textwidth]{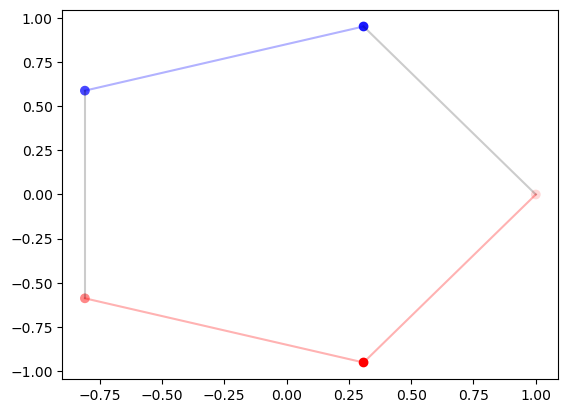}
		\includegraphics[width=0.3\textwidth]{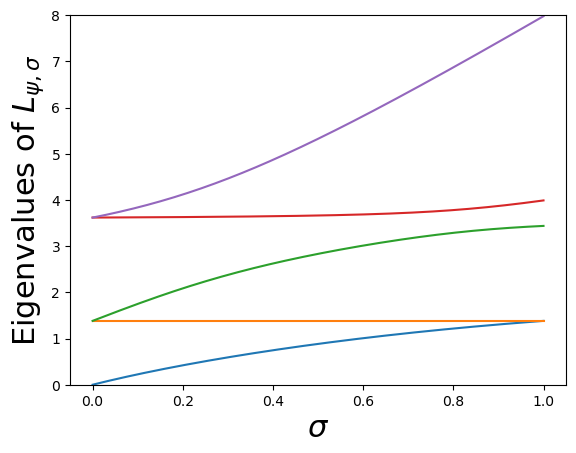}
		\includegraphics[width=0.3\textwidth]{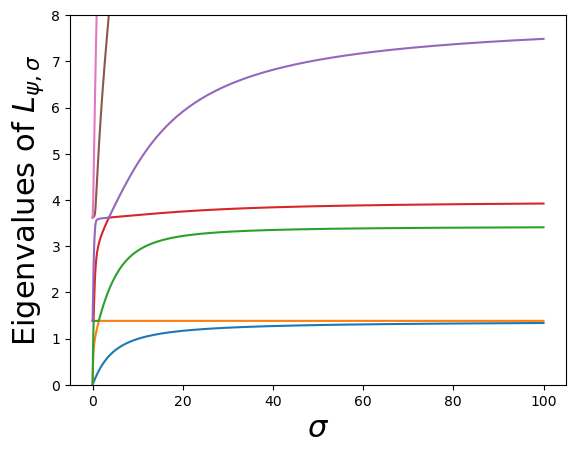}
		
		\caption{The edge-based (middle) and vertex-based (right) spectral flows for an eigenvector of the graph Laplacian for $C_5$ (left).}
		\label{fig:C5eig1}
	\end{figure}
	
	In \cref{fig:graph_nodal_defs} (top left), we consider the cyclic graph on 31 vertices $C_{31}$ and compute the number of nodal domains each eigenvector has. The dots correspond to pairs $(k, \nu(\phi(k))$, and dots are connected with a solid black line if the corresponding eigenvalues are the same. Since the eigenvectors all have the form $(1, \xi, \xi^2, ...,\xi^{30})$ with $\xi^{31} = 1$, taking the real and imaginary parts will produce two real valued eigenvectors with the same number of nodal domains. This is verifed in the scatter plot of nodal domains, where pairs of dots are connected by horizontal black lines.
	
	\begin{figure}
		\centering
		\includegraphics[width=0.3\textwidth]{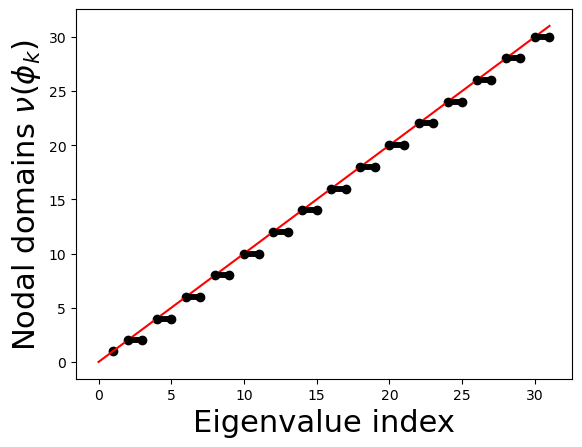}
		\includegraphics[width = 0.3\textwidth]{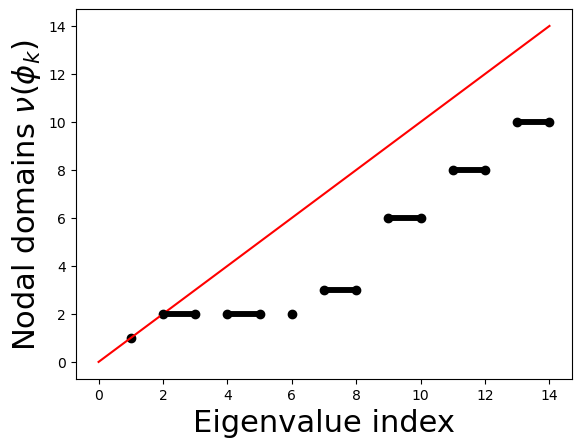}
		
		\includegraphics[width=0.3\textwidth]{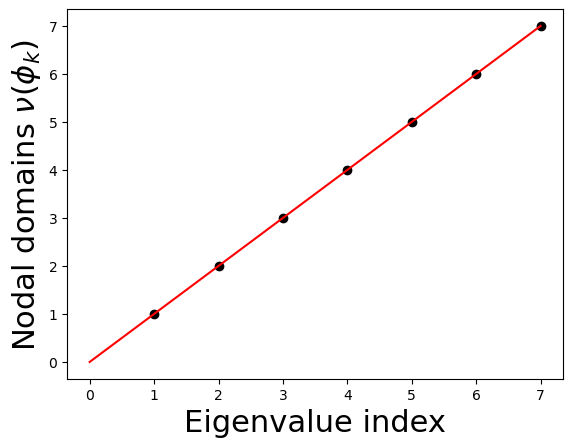}
		\includegraphics[width=0.3\textwidth]{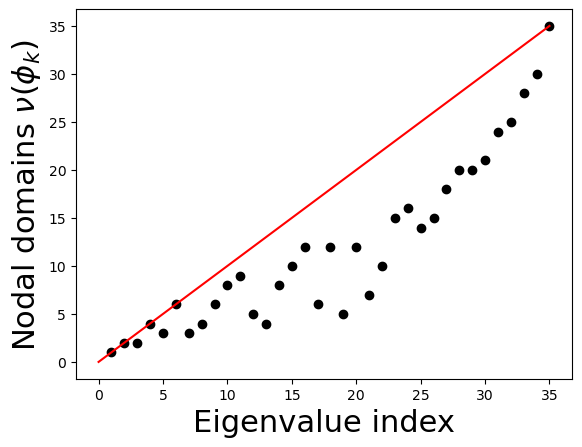}
		
		\caption{For the graphs $C_{31}$ (top left), $GP(7,3)$ (top right), $I_7$ (bottom left), $I_{7,5}$ (bottom right), plotted points correspond to pairs $(k, \nu(\phi_k))$ for eigenvectors $\phi_k$ of the graph's Laplacian, and black dots are connected by a line if they correspond to the same eigenvalue. The red line is the curve $y=x$; an eigenvector's nodal deficiency is the vertical distance between the corresponding dot and the red line.}
		\label{fig:graph_nodal_defs}
	\end{figure}
	
	\subsection{Petersen graphs}
	\label{section:examples-petersen}
	
	A generalized Petersen graph $GP(n,m)$ for $n\geq 3$ and $1\leq m \leq \lfloor \frac{n-1}{2} \rfloor$ consists of $2n$ vertices $\{a_0, ..., a_{n-1}, b_0, ..., b_{n-1}\}$, with edges of the form $(a_i, a_{i+1}), (a_i,b_i),$ and $(b_i, b_{i+m})$ for $0\leq i \leq n-1$, where the sums are considered modulo $n$. Some basic properties of these graphs are described in \cite{GeSt:SpectraGPetersen}.
	
	In \cref{fig:Petersen_flows}, we show the edge-based and vertex-based spectral flows for the 7th and 8th eigenvectors of $GP(7,3)$. Examining the two plots of the eigenvectors (top row), we count 3 nodal domains for each. However, both the edge-based (middle column) and vertex-based (right column) spectral flows seem to suggest that there should be 4 nodal domains, since 4 eigenvalue branches converge to $\lambda_* \approx 2.915$. When we zoom in to the edge-based spectral flow of the 7th eigenvector near $\sigma = 1$ for the edge-based flow and $\sigma = 600$ for the vertex-based flow (\cref{fig:Petersen_avoided_crossing}), we see that a crossing does in fact occur, meaning the final nodal domain count is actually 3. This example suggests that the interplay between the numerics and analysis of the spectral flow is more subtle than we might expect, since crossings in the edge-flow may occur close to the limit $\sigma = 1$ and converge to a value close to $\lambda_*$. Also note that in the vertex-based flows, only eigenvalue branches coming from ghost points converge to $\lambda_*$ from below; all other eigenvalue branches, especially those from $\spec(L)$, cross $\lambda_*$. In general, we have that if $|E_{\pm}| > |V|$, then all of the limiting eigenvalues originate from ghost vertices. Otherwise, some of the limiting eigenvalues may be branches from eigenvalues of $L$, depending on the nodal count and $|E_{\pm}|$.
	
	Finally, \cref{fig:graph_nodal_defs} (top right) displays the nodal domain counts for each eigenvector of $GP(7,3)$. Note that eigenvectors 7 and 8 have 3 nodal domains each, as verified by \cref{fig:Petersen_flows}.
	
	\begin{figure}
		\centering
		\includegraphics[width = 0.35\textwidth]{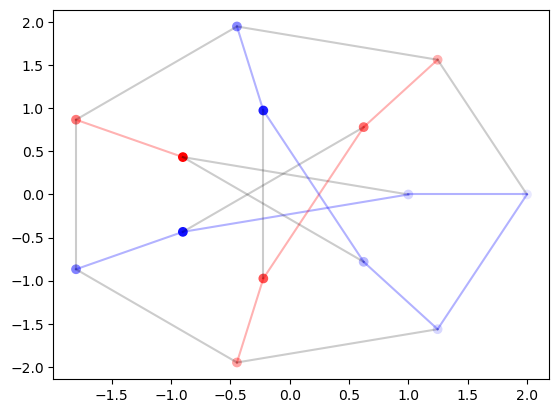}
		\includegraphics[width = 0.3\textwidth]{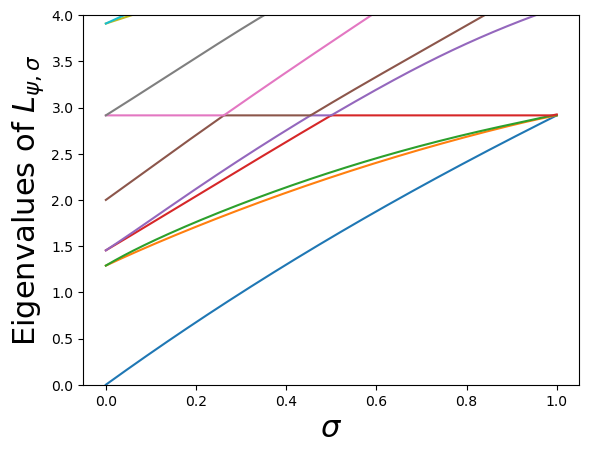}
		\includegraphics[width = 0.3\textwidth]{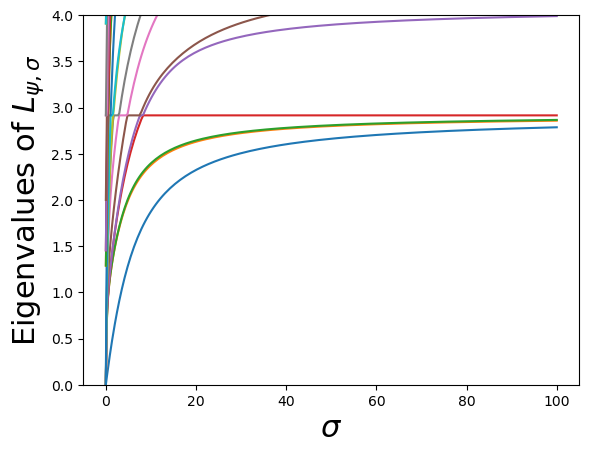}
		
		\includegraphics[width = 0.35\textwidth]{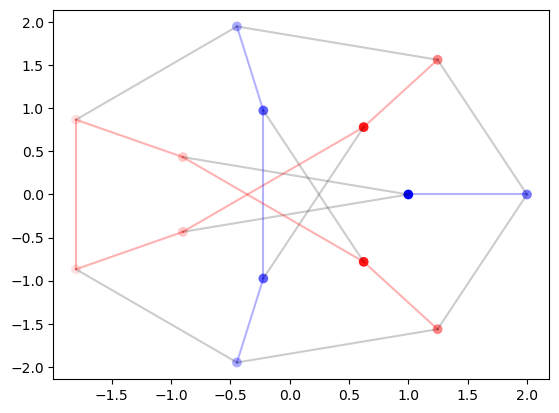}
		\includegraphics[width = 0.3\textwidth]{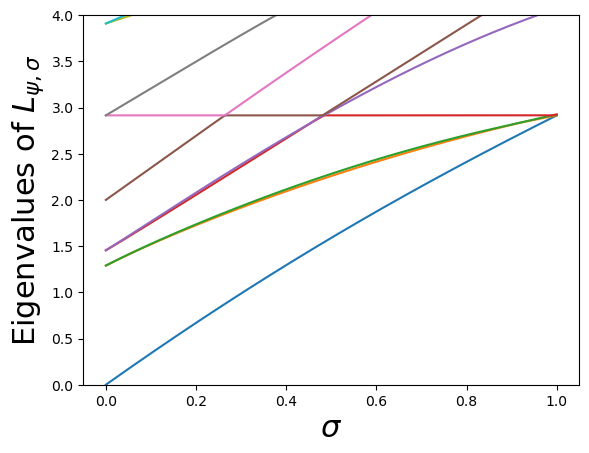}
		\includegraphics[width = 0.3\textwidth]{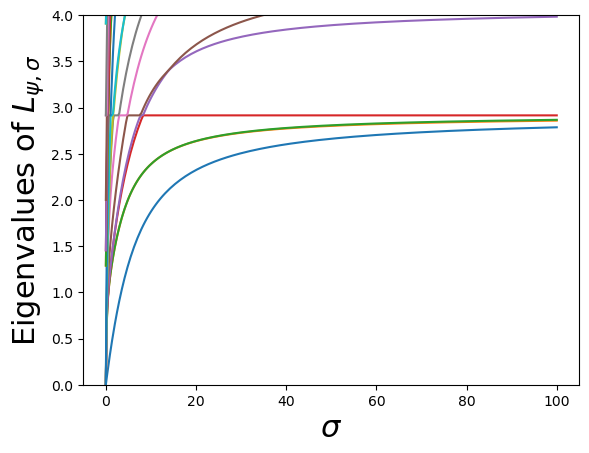}
		
		\caption{The 7th (top left) and 8th (bottom left) eigenvectors for the graph Laplacian of $GP(7,3)$, with their edge-based (middle column) and vertex-based (right column) flows. Note that, by \cref{fig:graph_nodal_defs} (top right), the corresponding eigenvalues are equal.}
		\label{fig:Petersen_flows}
	\end{figure}
	
	\begin{figure}
		\centering
		\includegraphics[width = 0.3\textwidth]{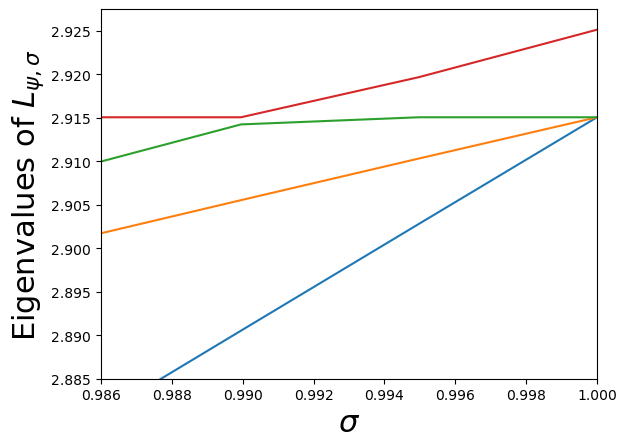}
		\includegraphics[width = 0.3\textwidth]{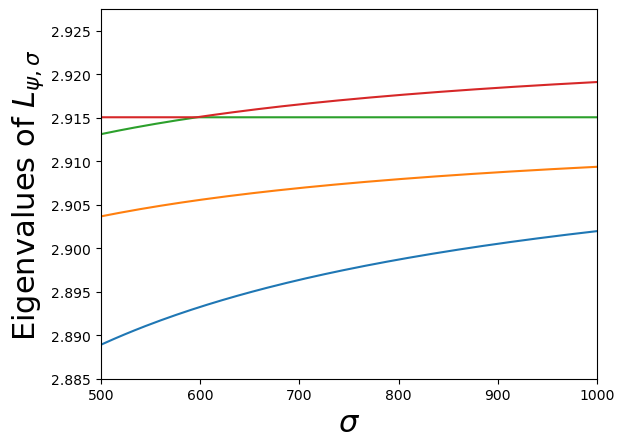}
		
		\caption{The crossing in the edge-based and vertex-based flow of the 7th eigenvector of $GP(7,3)$. In the edge-based flow the crossing occurs near $\sigma = 0.990$, while in the vertex-based flow the crossing occurs near $\sigma = 600$. The edge-based flow indicates the final nodal count is 3, not 4 as when examined from afar.}
		\label{fig:Petersen_avoided_crossing}
	\end{figure}
	
	\subsection{$1$- and $2$-d intervals}
	\label{section:examples-intervals}
	
	In this section we consider interval graphs $I_n$, and graph analogues of rectangles $I_{n,m}$. The vertices of $I_n$ are $\{1,2,...,n\}$, and the edges are $(i,i+1)$ for $1\leq i \leq n-1$. For $I_{n,m}$, we have vertices $\{v_{1,1}, ..., v_{1,m},v_{2,1}, ..., v_{n,m}\},$ and edges of the form $(v_{i,j},v_{i+1,j})$ and $(v_{i,j}, v_{i,j+1})$ for all possible $i$ and $j$. 
	
	The spectrum of $I_n$ is well-known, and is $\{2 - 2\cos(\frac{j \pi}{n})\}_{j=0}^{n-1}$; a simple argument involving a ``doubled'' interval and the spectrum of $C_{2n}$ is given in \cite{BrHa:SpecBook}. In \cref{fig:graph_nodal_defs} we show the nodal domain count of $I_7$ (bottom left), and in \cref{ fig:I_7_eig3_specflow} we show the spectral flow for the third eigenvector of $I_7$. Note that, as suggested by the continuum Sturm-Liouville theory, the eigenvectors of $I_7$ have zero nodal deficiency, since three eigenvalue branches converge to $\lambda_3$. In the vertex-based flow, the two eigenvalue branches converging to $\lambda_3$ come from ghost points, and all other eigenvalue branches cross $\lambda_3$.
	
	For the spectrum of $I_{n,m}$, we can take two eigenvectors $\phi_k, \psi_j$ of $I_k, I_j$, with corresponding eigenvalues $\lambda_k,\lambda'_j$, and define a Laplace eigenvector $\phi_k \otimes \psi_j$ on $I_{n,m}$ with eigenvalue $\lambda_k\lambda_j'$. These new eigenvectors are orthogonal to each other, and there are $nm$ of them, so we explicitly construct the eigenspaces of $I_{n,m}$'s Laplacian; that we also get the corresponding eigenvalues is a bonus. Since the eigenvectors of $I_{n,m}$ are all possible (outer) products of eigenvectors of $I_n$ and $I_m$, we cannot expect that the nodal deficiency is always zero. \Cref{fig:graph_nodal_defs} (bottom right) confirms this, which displays the number of nodal domains for each eigenvector of $I_{7,5}$. In \cref{fig:I_7_5_eig5_spec_flow}, we display a 3D visualization of the fifth eigenvector of $I_{7,5}$, together with its spectral flows zoomed in near $\lambda_5$. In each case 3 eigenvalue branches converge to $\lambda_5$, verifying what is observable from the eigenvector plot.
	
	\begin{figure}
		\centering
		\includegraphics[width=0.35\textwidth]{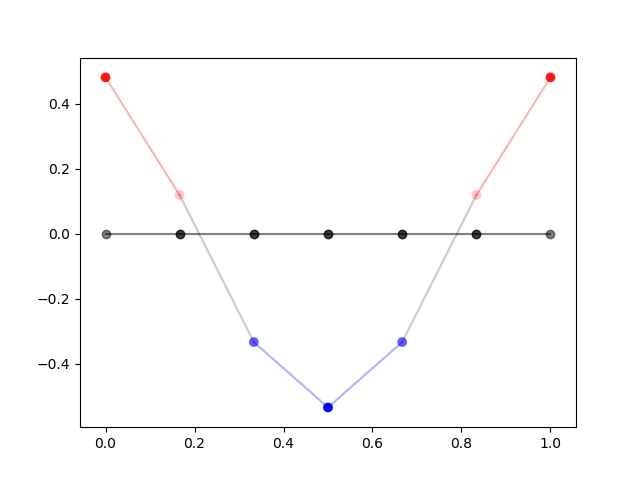}
		\includegraphics[width=0.3\textwidth]{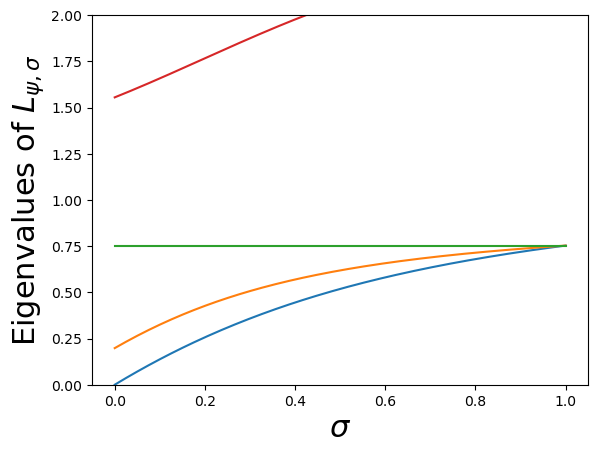}
		\includegraphics[width=0.3\textwidth]{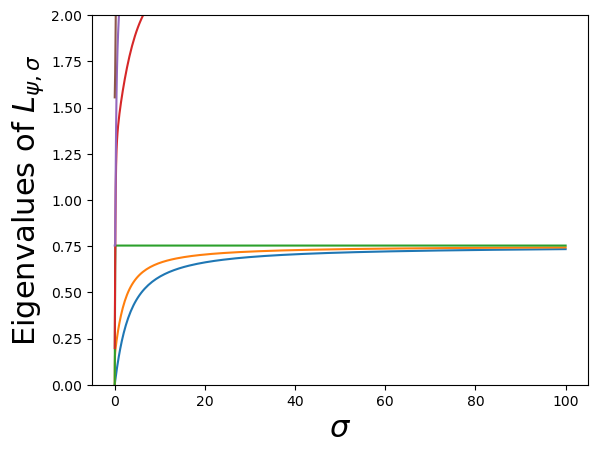}
		
		\caption{The spectral flow for the third eigenvector of the interval graph $I_7$. We display the eigenvector (left), along with its edge-based (middle) and vertex-based (right) spectral flows. The eigenvector graph shows three nodal domains, and each of the spectral flows have three eigenvalue branches converging to $\lambda_3 = 2 (1-\cos(\frac{2\pi}{7}))\approx 0.753$.} 
		\label{fig:I_7_eig3_specflow}
	\end{figure}
	
	\begin{figure}
		\centering
		\includegraphics[width=0.35\textwidth]{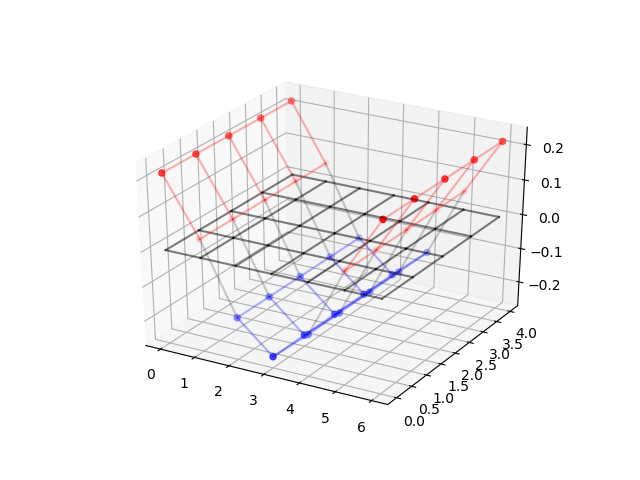}
		\includegraphics[width=0.3\textwidth]{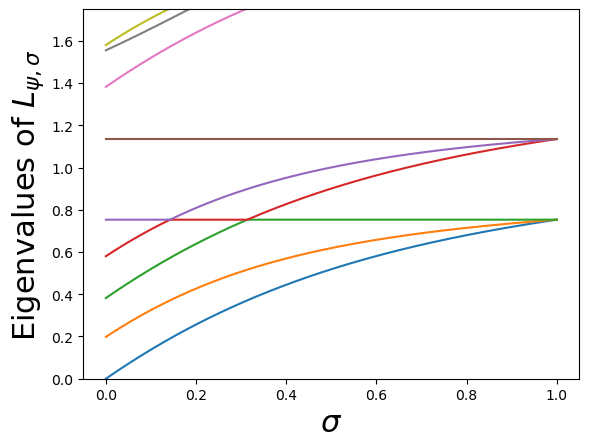}
		\includegraphics[width=0.3\textwidth]{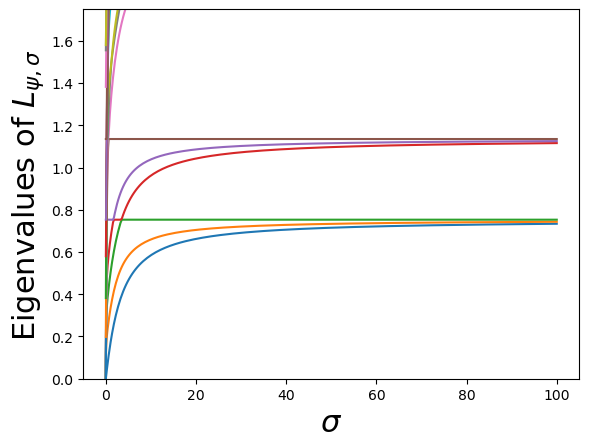}
		
		\caption{The spectral flow for the fifth eigenvector of the interval graph $I_{7,5}$. The eigenvector is displayed (left) along with the edge-based (middle) and vertex-based (right) spectral flows. This eigenvector has three nodal domains, and both the edge- and vertex-based flows have $3$ eigenvalues less than or equal to $\lambda_5$ in the limit.} 
		\label{fig:I_7_5_eig5_spec_flow}
	\end{figure}
	
	\subsection{Erd\H{o}s-R\'{e}nyi random graphs}
	\label{section:examples-random}
	
	In this subsection we explore some numerical examples involving Erd\H{o}s-R\'{e}nyi graphs, in which each edge between vertices is chosen with probability $p$; these graphs are often denoted $G(n,p)$ where $n$ is the number of vertices. We sampled a graph $G(20,p)$ for $p=0.1, 0.3, 0.5, 0.7, 0.9, 0.95$, and the number of nodal domains for each sample's eigenvectors were computed. \Cref{fig:ER_graphs} shows these samples with their third eigenvector plotted (first row, third row), along with the scatter plots showing the corresponding nodal domain counts for all eigenvectors (second row, fourth row). The nodal domain counts suggest that for smaller edge probabilities the random graphs more closely resemble intervals, whereas for higher edge probabilities the random graphs more closely resemble complete graphs.
	
	\begin{figure}
		\centering
		\includegraphics[width=0.3\textwidth]{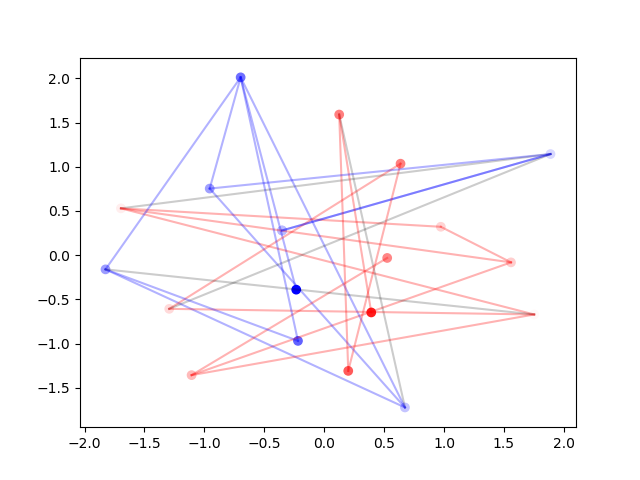}
		\includegraphics[width=0.3\textwidth]{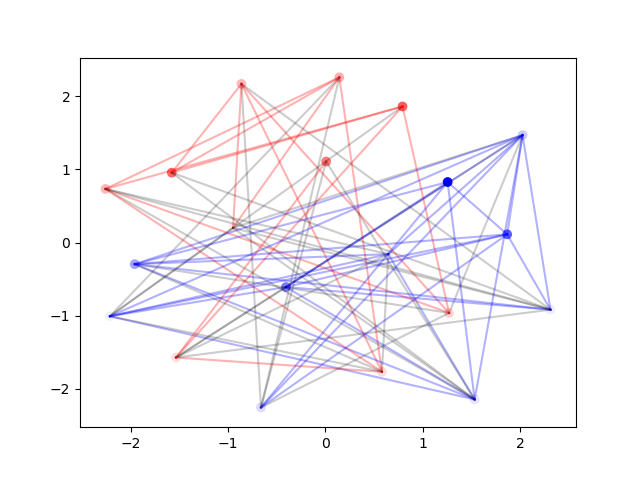}
		\includegraphics[width=0.3\textwidth]{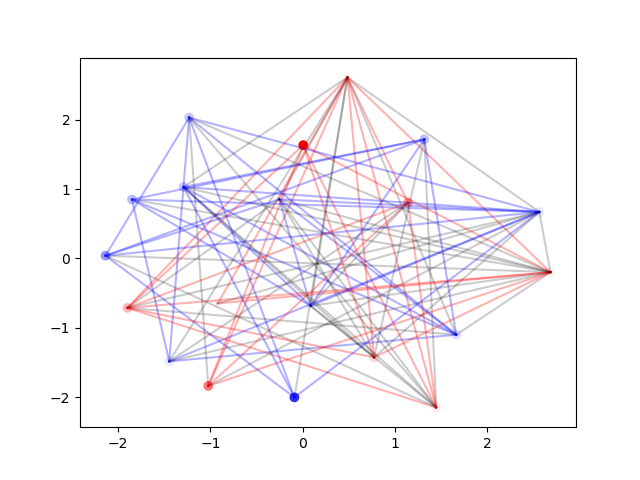}
		
		\includegraphics[width=0.3\textwidth]{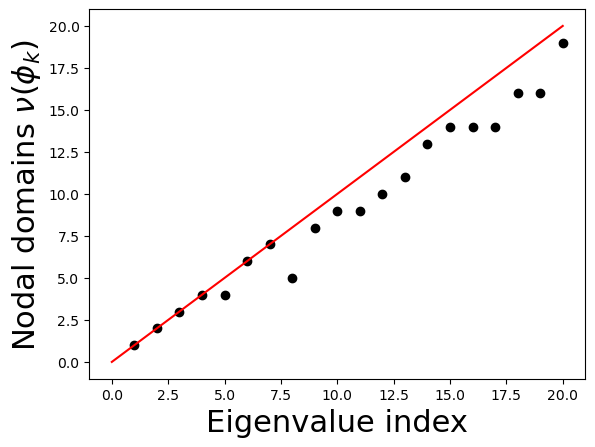}
		\includegraphics[width=0.3\textwidth]{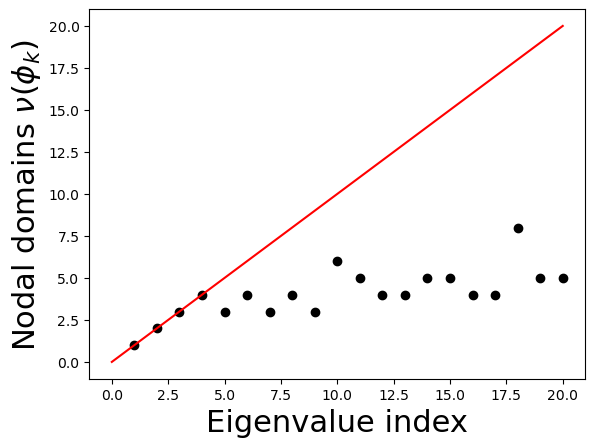}
		\includegraphics[width=0.3\textwidth]{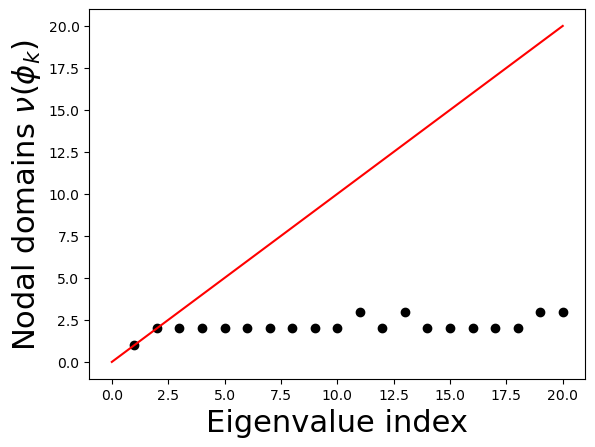}

		\includegraphics[width=0.3\textwidth]{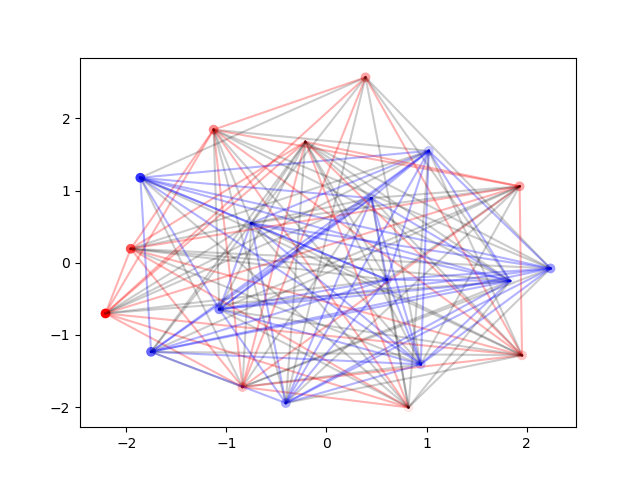}
		\includegraphics[width=0.3\textwidth]{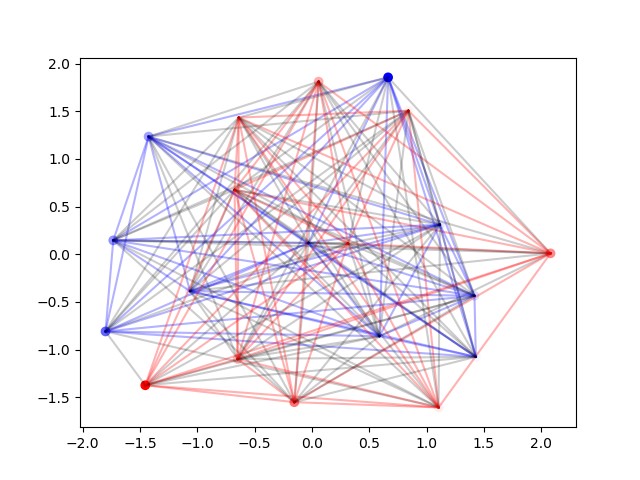}
		\includegraphics[width=0.3\textwidth]{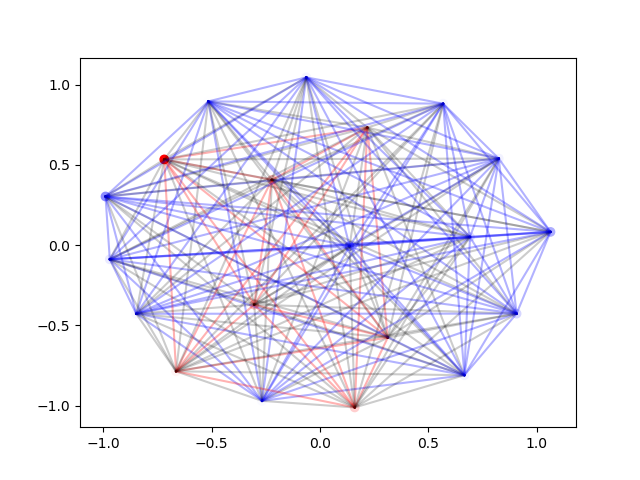}

		\includegraphics[width=0.3\textwidth]{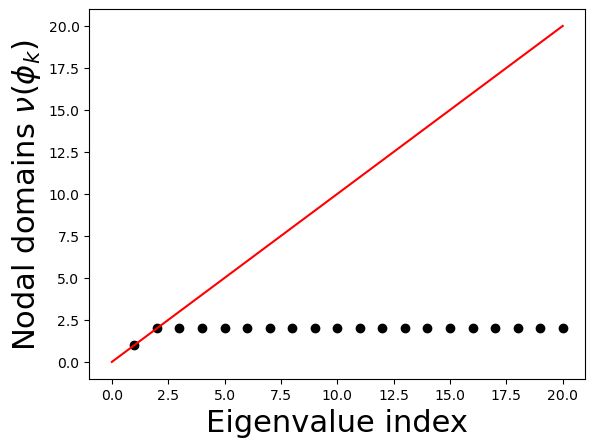}
		\includegraphics[width=0.3\textwidth]{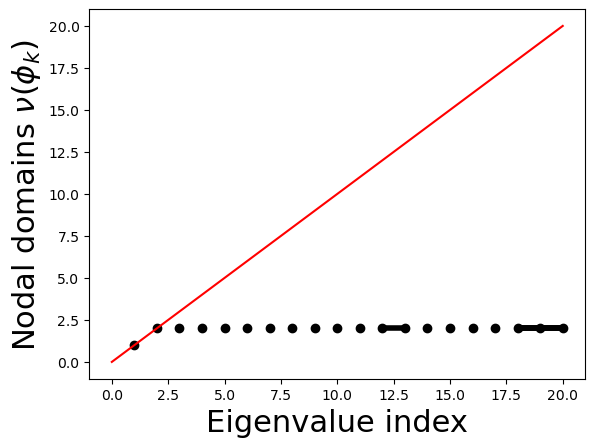}
		\includegraphics[width=0.3\textwidth]{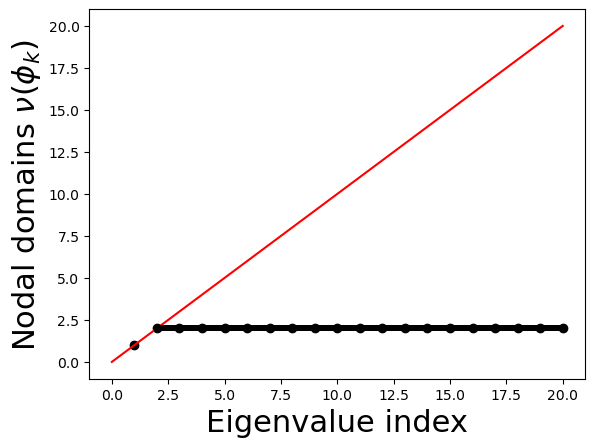}
		
		\caption{Various Erd\H{o}s-R\'{e}nyi random graphs on 20 vertices with their nodal domain counts; the graphs correspond to edge probabilities (left to right) $p=0.1,0.3,0.5$ in the top two rows and $p=0.7,0.9,0.95$ in the bottom two rows. As the probability that an edge connects vertices increases, the spectral flow is able to detect that they are closer to being a complete graph than an interval.}
		\label{fig:ER_graphs}
	\end{figure}

	\bibliography{graph_spec_flow}

\begin{thebibliography}{10}

\bibitem{Be:MatrixBook}
{\sc R.~Bellman}, {\em Introduction to matrix analysis}, vol.~19 of Classics in
  Applied Mathematics, Society for Industrial and Applied Mathematics (SIAM),
  Philadelphia, PA, 1997, \url{https://doi.org/10.1137/1.9781611971170},
  \url{https://doi-org.libproxy.lib.unc.edu/10.1137/1.9781611971170}.
\newblock Reprint of the second (1970) edition, With a foreword by Gene Golub.

\bibitem{Be:LowerBound}
{\sc G.~Berkolaiko}, {\em A lower bound for nodal count on discrete and metric
  graphs}, Comm. Math. Phys., 278 (2008), pp.~803--819,
  \url{https://doi.org/10.1007/s00220-007-0391-3},
  \url{https://doi-org.libproxy.lib.unc.edu/10.1007/s00220-007-0391-3}.

\bibitem{Be:NodalMag}
{\sc G.~Berkolaiko}, {\em Nodal count of graph eigenfunctions via magnetic
  perturbation}, Anal. PDE, 6 (2013), pp.~1213--1233,
  \url{https://doi.org/10.2140/apde.2013.6.1213},
  \url{https://doi-org.libproxy.lib.unc.edu/10.2140/apde.2013.6.1213}.

\bibitem{BeMaCo:SpecFlow}
{\sc G.~Berkolaiko, G.~Cox, and J.~L. Marzuola}, {\em Nodal deficiency,
  spectral flow, and the {D}irichlet-to-{N}eumann map}, Lett. Math. Phys., 109
  (2019), pp.~1611--1623, \url{https://doi.org/10.1007/s11005-019-01159-x},
  \url{https://doi-org.libproxy.lib.unc.edu/10.1007/s11005-019-01159-x}.

\bibitem{BiLeSt:EigvecBook}
{\sc T.~B{\i}y{\i}ko\u{g}lu, J.~Leydold, and P.~F. Stadler}, {\em Laplacian
  eigenvectors of graphs}, vol.~1915 of Lecture Notes in Mathematics, Springer,
  Berlin, 2007, \url{https://doi.org/10.1007/978-3-540-73510-6},
  \url{https://doi-org.libproxy.lib.unc.edu/10.1007/978-3-540-73510-6}.
\newblock Perron-Frobenius and Faber-Krahn type theorems.

\bibitem{BrHa:SpecBook}
{\sc A.~E. Brouwer and W.~H. Haemers}, {\em Spectra of graphs}, Universitext,
  Springer, New York, 2012, \url{https://doi.org/10.1007/978-1-4614-1939-6},
  \url{https://doi-org.libproxy.lib.unc.edu/10.1007/978-1-4614-1939-6}.

\bibitem{Ch:EigBook}
{\sc I.~Chavel}, {\em Eigenvalues in {R}iemannian geometry}, vol.~115 of Pure
  and Applied Mathematics, Academic Press, Inc., Orlando, FL, 1984.
\newblock Including a chapter by Burton Randol, With an appendix by Jozef
  Dodziuk.

\bibitem{Ch:SpecBook}
{\sc F.~R.~K. Chung}, {\em Spectral graph theory}, vol.~92 of CBMS Regional
  Conference Series in Mathematics, Published for the Conference Board of the
  Mathematical Sciences, Washington, DC; by the American Mathematical Society,
  Providence, RI, 1997.

\bibitem{dV:MagPaper}
{\sc Y.~Colin~de Verdi\`ere}, {\em Magnetic interpretation of the nodal defect
  on graphs}, Anal. PDE, 6 (2013), pp.~1235--1242,
  \url{https://doi.org/10.2140/apde.2013.6.1235},
  \url{https://doi-org.libproxy.lib.unc.edu/10.2140/apde.2013.6.1235}.

\bibitem{davies.2001.LAA}
{\sc E.~B. Davies, G.~M.~L. Gladwell, J.~Leydold, and P.~F. Stadler}, {\em
  Discrete nodal domain theorems}, Linear Algebra Appl., 336 (2001),
  pp.~51--60, \url{https://doi.org/10.1016/S0024-3795(01)00313-5},
  \url{https://doi-org.libproxy.lib.unc.edu/10.1016/S0024-3795(01)00313-5}.

\bibitem{Ev:PDEBook}
{\sc L.~C. Evans}, {\em Partial differential equations}, vol.~19 of Graduate
  Studies in Mathematics, American Mathematical Society, Providence, RI,
  second~ed., 2010, \url{https://doi.org/10.1090/gsm/019},
  \url{https://doi-org.libproxy.lib.unc.edu/10.1090/gsm/019}.

\bibitem{Fi:AcycMat}
{\sc M.~Fiedler}, {\em Eigenvectors of acyclic matrices}, Czechoslovak Math.
  J., 25(100) (1975), pp.~607--618.

\bibitem{GeSt:SpectraGPetersen}
{\sc R.~Gera and P.~St\u{a}nic\u{a}}, {\em The spectrum of generalized
  {P}etersen graphs}, Australas. J. Combin., 49 (2011), pp.~39--45.

\bibitem{Ka:PertBook}
{\sc T.~Kato}, {\em Perturbation theory for linear operators}, Classics in
  Mathematics, Springer-Verlag, Berlin, 1995.
\newblock Reprint of the 1980 edition.

\bibitem{Pl:Nodal}
{\sc A.~k. Pleijel}, {\em Remarks on {C}ourant's nodal line theorem}, Comm.
  Pure Appl. Math., 9 (1956), pp.~543--550,
  \url{https://doi.org/10.1002/cpa.3160090324},
  \url{https://doi-org.libproxy.lib.unc.edu/10.1002/cpa.3160090324}.

\bibitem{ReSi:Book4}
{\sc M.~Reed and B.~Simon}, {\em Methods of modern mathematical physics. {IV}.
  {A}nalysis of operators}, Academic Press [Harcourt Brace Jovanovich,
  Publishers], New York-London, 1978.

\end{thebibliography}

\end{document}